\documentclass[11pt]{amsart}

\usepackage[ansinew]{inputenc}
\usepackage[all]{xy}
\SelectTips{cm}{}
\usepackage[pagebackref,
  ,pdfborder={0 0 0}
  ,urlcolor=black,hypertexnames=false]{hyperref}
\hypersetup{pdfauthor={Clara L\"oh, Marco Moraschini and George Raptis}
  ,pdftitle=On the simplicial volume and the Euler characteristic of (aspherical) manifolds}

\usepackage{amsfonts}
\usepackage{amssymb, amsmath,amsthm, mathtools}
\usepackage{tabularx, hyperref}
\usepackage{enumerate}
\usepackage{mathrsfs}

\usepackage{tikz-cd}

\newtheorem{lemma}{Lemma}[section]
\newtheorem{thm}[lemma]{Theorem}
\newtheorem*{thm*}{Theorem}
\newtheorem{prop}[lemma]{Proposition}
\newtheorem*{prop*}{Proposition}
\newtheorem{cor}[lemma]{Corollary}

\newtheorem{conj}[lemma]{Conjecture}
\theoremstyle{definition}
\newtheorem{defi}[lemma]{Definition}

\newtheorem{quest}[lemma]{Question}
\newtheorem*{quest*}{Question}

\newtheorem{example}[lemma]{Example}

\newtheorem{rem}[lemma]{Remark}

\theoremstyle{definition}

\makeatletter
\newcommand\norm{\bBigg@{0.8}}

\makeatother
 
 \newcommand{\indnorm}[2][flex]{\csname #1l\endcsname\|#2%
                                 \csname #1r\endcsname\|\mathclose{}}
                                  \newcommand{\indnorml}[4][flex]{\csname #1l\endcsname\|#2%
                                 \csname #1r\endcsname\|_{#3}^{#4}\mathclose{}}
\newcommand{\sv}[2][flex]{\indnorm[#1]{#2}}
\newcommand{\isv}[2][norm]{\indnorml[#1]{#2}{\Z}{}}
\newcommand{\stisv}[2][norm]{\indnorm[#1]{#2}_\Z^\infty}
\newcommand{\svlf}[2][flex]{\indnorm[#1]{#2}^{\lfop}}
\DeclareMathOperator{\lfop}{lf}
\DeclareMathOperator{\cd}{cd}

\def\ltb#1#2{%
  b_{#1}^{(2)}(#2)}

\DeclareMathOperator{\catop}{cat}

\DeclareMathOperator{\Am}{Am}
\DeclareMathOperator{\amcat}{\catop_{\Am}}

\DeclareMathOperator{\lex}{\textup{Lex}}

\DeclareMathOperator{\comp}{comp}

\DeclareMathOperator{\im}{im}

\DeclareMathOperator{\id}{id}

\DeclareMathOperator{\SO}{SO}

\DeclareMathOperator{\Fill}{Fill}

\DeclareMathOperator{\Fillchi}{\Fill_\chi}

\DeclareMathOperator{\interior}{int}
\DeclareMathOperator{\ubc}{UBC}

\newcommand{\fa}[1]{%
  \forall_{#1}\quad}

\newcommand{\qand}{%
  \qquad\text{and}\qquad}

\DeclareMathOperator{\Mfd}{Mfd}
\def\Mon#1{\Mfd_{#1}^{\#}}
\DeclareMathOperator{\Cob}{Cob}
\def\Amcob#1{\Cob^{\Am}_{#1}}
\def\Gcob#1{\Cob^G_{#1}}
\DeclareMathOperator{\SV}{SV}

\newcommand{\N}{\ensuremath {\mathbb{N}}}
\newcommand{\R} {\ensuremath {\mathbb{R}}}
\newcommand{\Q} {\ensuremath {\mathbb{Q}}}
\newcommand{\Z} {\ensuremath {\mathbb{Z}}}

\renewcommand{\rho}{\varrho}
\def\phi{\varphi}

\def\args{\;\cdot\;}

\def\actson{\curvearrowright}

\def\plabel#1#2{%
  \tag{#1}%
  \label{#2}}
\def\pref#1{%
  Property~\eqref{#1}}

\usepackage{color}
\usepackage{pdfcolmk}

\long\def\forget#1{}

\def\longrightarrow{\rightarrow}

\def\emptyset{\varnothing}

\makeatletter
\@namedef{subjclassname@2020}{%
  \textup{2020} Mathematics Subject Classification}
\makeatother

\begin{document}

\title[Simplicial Volumes and Euler characteristics]{On the simplicial volume and the Euler characteristic of (aspherical) manifolds}

\author[]{Clara L\"oh}
\address{Fakult\"{a}t f\"{u}r Mathematik, Universit\"{a}t Regensburg, Regensburg, Germany}
\email{clara.loeh@ur.de}

\author[]{Marco Moraschini}
\address{Dipartimento di Matematica, Universit{\`a} di Bologna, 40126 Bologna, Italy}
\email{marco.moraschini2@unibo.it}

\author[]{George Raptis}
\address{Fakult\"{a}t f\"{u}r Mathematik, Universit\"{a}t Regensburg, Regensburg, Germany}
\email{georgios.raptis@ur.de}

\thanks{}

\date{\today.\ \copyright{C.~L\"oh, M.~Moraschini and G.~Raptis}.
  This work was supported by the CRC~1085 \emph{Higher Invariants}
  (Universit\"at Regensburg, funded by the~DFG)}

\begin{abstract}
  A well-known question by Gromov asks whether the vanishing of the
  simplicial volume of oriented closed aspherical manifolds
  implies the vanishing of the Euler characteristic.
  
  We study various versions of Gromov's question and collect
  strategies towards affirmative answers and strategies towards
  negative answers to this problem. Moreover, we put Gromov's question
  into context with other open problems in low- and high-dimensional
  topology.

  A special emphasis is put on a comparative analysis of the
  additivity properties of the simplicial volume and the Euler
  characteristic for manifolds with boundary. We explain that the
  simplicial volume defines a symmetric monoidal functor (TQFT) on the \emph{amenable} cobordism category, but not on the
  whole cobordism category.  In addition, using known computations of
  simplicial volumes, we conclude that the fundamental group of the
  $4$-dimensional amenable cobordism category is not finitely
  generated.

  We also consider new variations of Gromov's question. Specifically,
  we show that counterexamples exist among aspherical spaces that are
  only homology equivalent to oriented closed connected manifolds.
\end{abstract}
\maketitle

\setcounter{tocdepth}{1}
\tableofcontents

\section{Introduction}

The simplicial volume~$\sv M$ is a homotopy invariant of oriented
compact manifolds~$M$, defined as the $\ell^1$-semi-norm of the
singular $\R$-fundamental class. The simplicial volume is proportional
to the Riemannian volume for hyperbolic manifolds and zero in the
presence of amenability.  By the Gau\ss--Bonnet theorem and F\o lner
covering towers, a similar behaviour is also exhibited by the Euler
characteristic of aspherical manifolds. But understanding the
connection between the vanishing behaviour of the simplicial volume
and the Euler characteristic of closed aspherical manifolds remains a
mystery.

In particular, the following problem by Gromov is wide open:

\begin{quest}[{\cite[p.~232]{gromovasym}}]\label{quest:gromov}
  Let $M$ be an oriented closed aspherical manifold. Does
  the following implication hold?
  \begin{equation}
    \sv M = 0  \Longrightarrow \chi(M) = 0.
    \plabel{SV$\chi$}{p:svchi}
  \end{equation}
\end{quest}

The main challenge in answering Question~\ref{quest:gromov} is to find
a common ground for the various conditions and invariants involved:
asphericity, being a closed manifold, the vanishing of the simplicial
volume, and the vanishing of the Euler characteristic.

In this article, we explore different approaches to
Question~\ref{quest:gromov}, in search of both positive and negative
examples, as well as study its connections with other open problems in
low- and high-dimensional topology.

On the one hand, a key direction pursued in this paper is the
comparative study of the additivity properties of the simplicial
volume and of the Euler characteristic, which naturally leads us to
look at the simplicial volume (and stable integral simplicial volume)
of compact manifolds with boundary (see also Section~\ref{subsec:add}
below).

On the other hand, a different direction considered in this paper is
based on the observation that if the answer to
Question~\ref{quest:gromov} is affirmative, then \pref{p:svchi} will
hold for a general class of oriented closed manifolds which are only
homology equivalent to an aspherical one (see
Remark~\ref{key_rem}). Based on this, we consider generalized versions
of Question~\ref{quest:gromov} for aspherical spaces which are only
homology equivalent to an oriented closed manifold and for oriented
closed manifolds which are only homology equivalent to an aspherical
space (see also Section~\ref{subsec:heq} below).

A quick summary of various other strategies and examples is provided
in Section~\ref{subsec:relatedwork}.

\subsection{Additivity}\label{subsec:add}

As the category of closed aspherical manifolds is difficult to handle
structurally, we extend the setup to compact aspherical manifolds with
$\pi_1$-injective boundary (Section~\ref{subsec:relative}) and
consider an (equivalent) version of Gromov's question in this
context.

\begin{quest*}[see Question~\ref{quest:gromov:relative}]
 Let $(M, \partial M)$ be an oriented compact aspherical 
 manifold with non-empty $\pi_1$-injective aspherical boundary.
 Does the following implication hold?
  \begin{equation}
    \sv{M, \partial M} = 0  \Longrightarrow \chi(M, \partial M) = 0. 
    \plabel{SV$\chi$, $\partial$}{p:svchib}
  \end{equation}
\end{quest*}

We recall some examples and sufficient conditions for the vanishing of
the simplicial volume in Section~\ref{sec:van:sv}. Moreover, we show
an extension of Gromov's vanishing theorem (Theorem
\ref{thm:grom:vanishing:thm}) to the vanishing of the relative
simplicial volume $\sv{M, \partial M}$ for oriented compact connected
manifolds $(M, \partial M)$ which admit open covers with certain
amenability properties (Theorem~\ref{thm:relvan}).  Also, in
Section~\ref{sec:sisv}, we discuss the properties of the relative
stable integral simplicial volume and the analogues of the above
questions for this invariant.

\medskip

The extension to manifolds with boundary generally allows us to compare additivity and filling
properties of the simplicial volume and of the Euler characteristic more
systematically.  For example, we show that Question~\ref{quest:gromov} is related
to Edmonds' problem (Conjecture~\ref{conj:chi:1}) as follows:

\begin{prop*}[see Proposition~\ref{lemma:conj:chi:1:vs:Gromov}]
  Suppose that the following hold:
  \begin{itemize}
  \item[(a)] Every oriented closed aspherical $3$-manifold
    with amenable fundamental group is the $\pi_1$-injective boundary of an oriented
    compact aspherical $4$-manifold~$W$ with~$\sv{W,\partial W} = 0$.
  \item[(b)] All oriented closed aspherical $4$-manifolds
    satisfy \pref{p:svchi}.
  \end{itemize}  
  Then, there exists an oriented closed aspherical
  $4$-manifold~$M$ with~$\chi(M)=1$.
\end{prop*}

The notion of a topological quantum field theory (TQFT), as a
symmetric monoidal functor on the cobordism category, provides an
efficient way of encoding additivity properties. The connections
between the simplicial volume and (invertible) TQFTs are discussed in
Section~\ref{sec:tqft}.  Specifically, based on known additivity
properties, we explain in Section~\ref{subsec:tqft} that the
simplicial volume defines an invertible TQFT on a suitable
\emph{amenable cobordism category} with values in~$\R$. In addition,
we show that this functor cannot be extended to a functor on the whole
cobordism category (Proposition~\ref{prop:tqft}).  This contrasts the
additivity behaviour of the relative Euler characteristic, which is
unconditional and thus defines a TQFT on the whole
cobordism category (Remark~\ref{rem:eulerTQFT}).  We also obtain the
following result about the fundamental group of the amenable cobordism
category of $4$-manifolds and related cobordism categories (see
Section~\ref{subsec:tqft} for the precise definition of~$\Gcob d$):

\begin{thm*}[see Theorem~\ref{thm:cob}]\label{ithm:cob}
  Let $G$ be a class of amenable groups that is closed under
  isomorphisms and let $M$ be an object of~$\Gcob 4$. Then the
  group~$\pi_1(B \Gcob 4, [M])$ is \emph{not} finitely generated.
\end{thm*}

\subsection{Aspherical spaces homology equivalent to closed manifolds}\label{subsec:heq}

Using the Kan--Thurston theorem, we show in
Section~\ref{sec:asphericalisation} that \pref{p:svchi} fails in
general if closed aspherical manifolds are replaced by aspherical
spaces homology equivalent to closed manifolds or closed manifolds 
homology equivalent to an aspherical space (see Section~\ref{subsec:acyclicmaps}
for the definition of an acyclic map):

\begin{thm*}[see Theorem~\ref{thm:homologymfd}]
  Let $n \in \N_{\geq 2}$ be even. 
  \begin{enumerate}
  \item There exist aspherical spaces~$X$ that admit an acyclic map~$X
    \to M$ to an oriented closed connected $n$-manifold~$M$ and
    satisfy $\sv{X} = 0$ and $\chi(X) \neq 0$. In particular, these
    aspherical spaces do \emph{not} satisfy \pref{p:svchi}.
  \item There exist oriented closed connected $n$-manifolds~$M$ that
    admit an acyclic map~$X \to M$ from an aspherical space~$X$ and
    satisfy $\sv{M} = 0$ and $\chi(M) \neq 0$. In particular, these
    manifolds do \emph{not} satisfy \pref{p:svchi}.
\end{enumerate}
\end{thm*}

\subsection{Strategies and known examples}\label{subsec:relatedwork}

As the following examples show, the hypotheses in
Question~\ref{quest:gromov} cannot be reasonably weakened or modified
in any straightforward way:
\begin{itemize}
\item In general, non-aspherical oriented closed connected manifolds
  do not satisfy \pref{p:svchi}: For example, $\sv{S^2} = 0$,
  but~$\chi(S^2) = 2$.
\item The converse implication of \pref{p:svchi} does not hold in
  general for aspherical manifolds: For example, oriented closed
  connected hyperbolic $3$-manifolds have vanishing Euler
  characteristic, but their simplicial volume is non-zero.
\item
  In general, \pref{p:svchi} does not hold for oriented compact
  connected manifolds with non-empty boundary without imposing
  additional conditions on the inclusion of the boundary
  (Remark~\ref{rem:svchibfails}).
\item
  In general, \pref{p:svchi} does not hold for aspherical spaces that
  are only homology equivalent to oriented closed connected manifolds
  (
  Theorem~\ref{thm:homologymfd}).
\end{itemize}
We also refer to Section~\ref{sec:van:sv} for a suvey of examples of
oriented closed manifolds with vanishing or non-vanishing simplicial
volume.

\smallskip

Various strategies have been developed to handle
Question~\ref{quest:gromov}. In particular, this also led to a wide
range of positive examples:

\subsubsection*{Direct computations of both sides}

One of Gromov's original motivations to formulate
Question~\ref{quest:gromov} was the observation that the simplicial
volume and the Euler characteristic share some common vanishing
properties. Examples of this phenomenon include manifolds that admit
non-trivial self-coverings, closed aspherical manifolds with amenable
fundamental group, and closed aspherical manifolds that admit small
amenable open covers (Section~\ref{sec:van:sv}).

\subsubsection*{Boundedness properties of the Euler class}

The simplicial volume of an oriented closed connected $n$-manifold~$M$
is closely related to the comparison map from bounded cohomology to
singular cohomology in degree~$n$ (see Section~\ref{subsec:comp} and
Proposition~\ref{prop:duality}). On the other hand, the Euler
characteristic is related to the Euler class by duality. As we explain
in Proposition~\ref{prop:bounded:euler:gromov:quest}, an immediate
consequence is that \pref{p:svchi} can be reformulated in terms of the
boundedness of the Euler class.  The problem of the boundedness of the
Euler class is well studied and understood in several
cases~\cite{milnor_euler,wood_euler,ivanov_turaev,bucher_finiteness}.

\subsubsection*{$L^2$-Betti numbers}

Gromov suggested to use the fact that the Euler characteristic can be
computed as the alternating sum of the $L^2$-Betti numbers and asked the
following version of Question~\ref{quest:gromov}:

\begin{quest}[\protect{\cite[p.~232]{gromovasym}}]\label{quest:gromovL2}
  Let $M$ be an oriented closed aspherical manifold. Does the
  following implication hold?
  \begin{equation}
    \sv M = 0  \Longrightarrow  \bigl(\fa{k \in \N} \ltb k M = 0 \bigr).
    \plabel{SV$L^2$}{p:svL2}
  \end{equation}
\end{quest}

Assuming the Singer conjecture on the vanishing of the $L^2$-Betti numbers
of closed aspherical manifolds outside the middle
dimension~\cite{dodziuk}, Question~\ref{quest:gromov} and
Question~\ref{quest:gromovL2} are equivalent. More concretely,
Gromov~\cite[p.~306]{gromovmetric} proposed a definition of integral
foliated simplicial volume, involving dynamical systems, and then
\begin{itemize}
\item to establish an upper bound of the $L^2$-Betti numbers in terms of
  integral foliated simplicial volume (via Poincar\'e duality), and
\item to investigate whether the vanishing of the simplicial volume of
  closed aspherical manifolds implies the vanishing of
  integral foliated simplicial volume.
\end{itemize}
The first step has been carried out by Schmidt~\cite{Sthesis}. The
second step is an open problem, which is known to have a positive
answer in many cases, e.g., for oriented closed aspherical 
manifolds
\begin{itemize}
\item that have amenable fundamental group~\cite{FLPS},
\item that carry a non-trivial smooth $S^1$-action~\cite{FauserS1}, 
\item that are generalised graph manifolds~\cite{FFL},
\item that are smooth and have trivial minimal
  volume~\cite[(proof of) Corollary 5.4]{braunphd}.
  In particular, Question~\ref{quest:gromov} with the simplicial
  volume replaced by the minimal volume has a positive answer~\cite{sauerminvol}.
\end{itemize}
Moreover, the integral foliated simplicial volume is related to the
cost of the fundamental group~\cite{loeh_cost} and to the stable
integral simplicial volume~\cite{LP,FLPS} (Section~\ref{sec:sisv}).
In turn, the stable integral simplicial volume gives upper bounds for
homology growth, torsion homology growth~\cite{FLPS}, and the rank
gradient~\cite{loeh_rg}.

\subsubsection*{Functorial semi-norms}

If the integral foliated simplicial volume is a functorial semi-norm
on aspherical closed manifolds, then Question~\ref{quest:gromov}
has an affirmative answer~\cite[Theorem~2.2.2]{FauserT}.

\subsubsection*{Geometric positivity results} 

Conversely, it is known that many examples of closed aspherical 
manifolds with potentially non-zero Euler characteristic have positive
simplicial volume. Examples include oriented closed connected
hyperbolic manifolds~\cite{thurston,vbc}, closed manifolds with
negative sectional curvature~\cite{inoueyano}, closed irreducible
locally symmetric spaces of higher rank~\cite{lafontschmidt}, closed
manifolds with non-positive sectional curvature and sufficiently
negative intermediate Ricci curvature~\cite{connellwang}, and closed
manifolds with non-positive sectional curvature and strong enough
conditions at a single point~\cite{connellwang2}.  We refer to
Section~\ref{subsec:examples:van:sv} for further examples of manifolds
with (non-)vanishing simplicial volume.

\subsubsection*{Outlook}
Supported by the wealth of positive examples, and in view of the
existence of ``exotic'' aspherical manifolds, it seems plausible that 
Question~\ref{quest:gromov} has a positive answer in the following 
special case:

\begin{quest}
  Let $M$ be an oriented closed aspherical $n$-manifold
  whose universal covering is homeomorphic to~$\R^n$. Does $M$ satisfy
  \pref{p:svchi}?
\end{quest}

\subsection*{Organisation of the article}

In Section~\ref{sec:prelim}, we collect the definitions of simplicial
volume (Section~\ref{subsec:sv}) and bounded cohomology
(Section~\ref{subsec:comp}) as well as the duality principle which
connects these (Section~\ref{subsec:comp}).  Moreover, we discuss the
behaviour of the simplicial volume with respect to glueings
(Section~\ref{subsec:sv:glueings}) and introduce a relative version of
Question~\ref{quest:gromov} (Section~\ref{subsec:relative}).  Finally,
in Section~\ref{subsec:euler:class}, we discuss the boundedness
properties of the Euler class in connection with
Question~\ref{quest:gromov}.

Section~\ref{sec:van:sv} is mainly devoted to the vanishing of the
simplicial volume.  Some known examples are collected in
Section~\ref{subsec:examples:van:sv}.  Vanishing results for the
simplicial volume assuming the existence of amenable open covers are
recalled in Section~\ref{subsec:amen:cover:closed} and extended to manifolds with boundary in
Section~\ref{subsec:amen:cover:relative}.  Known results and open
problems about the behaviour of the simplicial volume with respect to
products are recalled in Section~\ref{subsec:products:sv} and these
are then discussed in connection with Question~\ref{quest:gromov}
(Proposition~\ref{prop:gromov:quest:products}).  Finally,
Section~\ref{subsec:filling:sv} explains a connection between
Question~\ref{quest:gromov} and a conjecture of Edmonds in
four-dimensional topology (Conjecture~\ref{conj:chi:1}) via
``fillings'' of closed manifolds.

In Section~\ref{sec:tqft}, we define the amenable cobordism category
and explain how to interpret the simplicial volume as an invertible
TQFT on this cobordism category. Using the simplicial volume, we prove 
that the fundamental group of the $4$-dimensional amenable cobordism category 
is not finitely generated (Theorem~\ref{thm:cob}). Also, using known results 
about cobordism categories, we show that the simplicial volume does not extend 
to the whole cobordism category (Proposition~\ref{prop:tqft}).

Section~\ref{sec:asphericalisation} is concerned with the study of
Question~\ref{quest:gromov} using known constructions that produce
aspherical spaces. More precisely, in
Sections~\ref{subsec:svhomologymfd}--\ref{subsec:kanthurston}, we
recall the Kan--Thurston theorem and explain how to use this to prove
the result stated in Section~\ref{subsec:heq} (see Theorem~\ref{thm:homologymfd}). Then, in
Section~\ref{subsec:other_aspher}, we briefly review known constructions of
closed aspherical manifolds, Davis' reflection group trick and Gromov's
hyperbolization, in the context of Question~\ref{quest:gromov}.

Finally, Section~\ref{sec:sisv} surveys the approach to
Question~\ref{quest:gromov} via the stable integral simplicial volume.

\subsection*{Notation}

We use~$\N = \{0,1,2, \dots \}$. We recall that aspherical spaces are assumed to be path-connected. 

\subsection*{Acknowledgements}

We would like to thank Igor Belegradek for helpful comments on
Section~\ref{subsec:other_aspher}.

\section{Simplicial volume}\label{sec:prelim}

We recall the definition of the simplicial volume, basic glueing
properties, and the role of the comparison map for bounded
cohomology. Moreover, we consider and study a relative version of
Question~\ref{quest:gromov}, and discuss the connection of
Question~\ref{quest:gromov} with (the boundedness of) the Euler class.

\subsection{Simplicial volume}\label{subsec:sv}

The simplicial volume originally appeared in Gromov's proof of Mostow
rigidity as a homotopy invariant replacement of the hyperbolic
volume~\cite{munkholm,vbc}.

\begin{defi}[Simplicial volume]
  Let $M$ be an oriented closed connected $n$-manifold. The
  \emph{simplicial volume} of~$M$ is defined as
  $$
  \sv{M} \coloneqq \sv{[M]}_1 \in \R_{\geq 0},
  $$
  where $[M] \in \, H_n(M; \R)$ is the $\R$-fundamental class of~$M$
  and $\|\cdot\|_1$ denotes the semi-norm on~$H_*(\args;\R)$, induced
  by the $\ell^1$-norm on the singular chain complex~$C_*(\args;\R)$
  with respect to the basis given by the singular simplices.
\end{defi}

\begin{defi}[Relative simplicial volume]
  Let $(M, \partial M)$ be an oriented compact connected
  $n$-manifold~$M$ with boundary~$\partial M$.  The \emph{relative
    simplicial volume} of~$(M, \partial M)$ is defined as
  $$
  \sv{M, \partial M} \coloneqq \sv{[M, \partial M]}_1 \in \R_{\geq 0},
  $$
  where $[M, \partial M] \in \, H_n(M, \partial M; \R)$ is the
  $\R$-fundamental class of~$(M, \partial M)$ and $\|\cdot\|_1$
  denotes the $\ell^1$-semi-norm on relative singular homology.
\end{defi}

In the oriented, compact, non-connected case, we define the (relative)
simplicial volume as the sum of the (relative) simplicial volumes of
the components. In particular, $\sv{\emptyset} = 0$.

\begin{rem}\label{rem:vanishing:rel:sv:implices:van:boundary}
  The boundary of a relative fundamental cycle of~$(M, \partial M)$ is
  a fundamental cycle of~$\partial M$.  This shows that for every
  oriented compact connected $n$-manifold~$M$ with non-empty
  boundary~$\partial M$, we have
  $$
  \sv{M, \partial M} \geq \frac{\sv{\partial M}}{n+1}.
  $$
  In particular, $\sv{M, \partial M} = 0$ implies $\sv{\partial M} = 0$.
  
  Note that in the case of compact $3$-manifolds better estimates
  are available~\cite{BFP3}.
\end{rem}

\begin{rem}\label{rem:sv:with:int:coeff}
  One can also define the (relative) simplicial volume with
  \emph{integral coefficients} just by working with integral singular
  homology.  More precisely, the \emph{integral (relative) simplicial
    volume} of an oriented compact connected $n$-manifold~$M$ with
  (possibly empty) boundary~$\partial M$ is defined by
  $$
  \sv{M, \partial M}_{\Z} \coloneqq \sv{[M, \partial M]_{\Z}}_1 \in \N,
  $$
  where $[M, \partial M]_\Z \in \, H_n(M, \partial M; \Z)$ is the
  $\Z$-fundamental class of~$(M, \partial M)$.  Notice that we still
  have $\sv{M, \partial M}_{\Z} \geq \sv{\partial M}_{\Z} \slash
  (n+1)$.
\end{rem}

\subsection{Simplicial volume and glueings of manifolds}\label{subsec:sv:glueings}

In general, the simplicial volume is not additive with respect to the
glueing of manifolds along submanifolds. However, (sub)additivity does
hold in the case of amenable glueings:

\begin{thm}[Simplicial volume and glueings~{{\cite{vbc, BBFIPP}}\cite[Theorem~7.6]{Frigerio}}]\label{thm:sv:additivity}
  Let $I$ be a finite set and let $(M_i, \partial M_i)_{i \in I}$ be a
  family of oriented compact connected manifolds of the same
  dimension. Assume that all the boundary components have 
  amenable fundamental group.
  Moreover, let $(M,\partial M)$ be obtained from~$(M_i,
  \partial M_i)_{i \in I}$ by a pairwise glueing (along orientation
  reversing homeomorphisms) of a set of
  boundary components. Then, we have
  \[ \sv{M,\partial M} \leq \sum_{i \in I} \sv{M_i,\partial M_i}.
  \]
  If all glued boundary components are $\pi_1$-injective in their
  original manifold, then
  \[ \sv{M,\partial M} = \sum_{i \in I} \sv{M_i,\partial M_i}.
  \]
\end{thm}

\begin{rem}
  Theorem~\ref{thm:sv:additivity} allows that boundary components of
  the given manifolds~$(M_i,\partial M_i)$ are glued to boundary
  components of the same manifold~$(M_i,\partial M_i)$; i.e.,
  self-glueings are included.  Furthermore, not all boundary
  components need to be glued (so that some of the components remain
  boundary components of~$M$).
\end{rem}

\begin{rem}\label{ex:additivity:Euler:characteristic}
  It is well-known that the Euler characteristic is \emph{always}
  additive with respect to glueings: given oriented compact connected
  $n$-manifolds $(M, \partial M)$ and~$(N, \partial N)$ with
  homeomorphic (or just homotopy equivalent) boundary components $M_1
  \subseteq \partial M$ and~$N_1 \subseteq \partial N$, we set $Z = M
  \cup_{M_1 \cong N_1} N$. Then we have:
  $$
  \chi(Z) = \chi(M) + \chi(N) - \chi(M_1 \cong N_1) 
  $$
  and similarly:
  $$
  \chi(Z, \partial Z) = \chi(M, \partial M) + \chi(N, \partial N) + \chi(M_1 \cong N_1).
  $$
  Here $\chi(W, \partial W) \coloneqq \chi(W) - \chi(\partial W)$ denotes the \emph{relative}
  Euler characteristic of the compact manifold $(W, \partial W)$.

  Assuming that $M_1$ is aspherical and has amenable fundamental
  group, then both the simplicial volume and the Euler characteristic
  of~$M_1$ vanish (Example~\ref{ex:zero:simplicial:volume}(1) and
  Theorem~\ref{thm:van:thm:euler:characteristic}).  In particular, the
  last formula simplifies in this case to a formula analogous to the
  one in Theorem~\ref{thm:sv:additivity}:
  $$
  \chi(Z, \partial Z) = \chi(M, \partial M) + \chi(N, \partial N).
  $$
\end{rem}

\begin{example}[Doubles]\label{exa:double:sv}
  Given an oriented compact connected manifold~$M$ with non-empty
  boundary~$\partial M$, we define the \emph{double} of~$M$ to be
  $$
  D(M) \coloneqq M \cup_{\partial M \cong \partial (-M)} -M,
  $$
  where $-M$ denotes a copy of~$M$ with the opposite orientation.  It
  is easily seen that we always have subadditivity of the simplicial
  volume in this case:
  $$
  \sv{D(M)} \leq 2 \cdot \sv{M, \partial M}.
  $$
  Indeed, given a relative fundamental cycle~$c$ of~$M$, we can
  set~$\overline {c}$ to be the relative fundamental cycle of~$-M$
  corresponding to~$-c$.  Then, $c' = c + \overline{c}$ is in fact a
  fundamental cycle of~$D(M)$ with norm
  $$
  |c'|_1 \leq |c|_1 + |\overline{c}|_1 = 2 \cdot |c|_1.
  $$ 
  Then the subadditivity of the simplicial volume follows from taking
  the infimum over all such~$c$.  The same computation also works for
  integral coefficients.
\end{example}

\begin{rem}[Doubles and asphericity]\label{rem:double:aspherical}
  In general, the double of an oriented compact aspherical
  manifold with boundary is not necessarily aspherical; prototypical
  examples of this kind are the $3$-ball or~$S^1 \times D^2$.

  Let $(M, \partial M)$ be a compact $n$-manifold, where $M$ is
  aspherical and $\partial M$ is connected, and let $F$ denote the
  homotopy fibre of the inclusion~$\partial M \subset M$. For
  simplicity, we write $H := \pi_1(\partial M, x)$ and $G:= \pi_1(M,
  x)$ and denote by $\iota \colon H \to G$ the induced homomorphism.
  Moreover, let $G' := \mathrm{im}(\iota)$ be the image of~$\iota$,
  and consider the corresponding diagram
  $$
  \xymatrix{
    K(G,1) \ar@{=}[d] & \partial M \ar[r] \ar[l] \ar[d]^{q} & K(G,1) \ar@{=}[d] \\
    K(G,1) & K(G',1) \ar[r] \ar[l] & K(G,1)
  }
  $$
  induced by~$\iota$ and the inclusion maps. Suppose that the
  (homotopy) pushout
  $$D(M) = M \cup_{\partial M} M \simeq K(G,1)
  \stackrel{h}{\cup}_{\partial M} K(G,1)$$ is aspherical.  Then the
  induced map between the homotopy pushouts
  $$g \colon K(G,1) \stackrel{h}{\cup}_{\partial M} K(G,1) \to K(G,1) \stackrel{h}{\cup}_{K(G',1)} K(G,1)$$
  is a homotopy equivalence by the Seifert--van~Kampen
  theorem. Passing to the homotopy fibres of the diagram above, regarded as
  a diagram over~$K(G,1)$, we obtain the following diagram (up to
  canonical homotopy equivalence)
  $$
  \xymatrix{
    \ast \ar@{=}[d] & F \ar[r] \ar[l] \ar[d]^{q'} & \ast \ar@{=}[d] \\
    \ast & D \ar[r] \ar[l] & \ast
  }
  $$
  where $D$ is discrete with cardinality equal to the
  index~$[G:G']$. We recall that the homotopy fibre of a map from a
  homotopy pushout is canonically identified up to homotopy
  equivalence with the homotopy pushout of the respective homotopy
  fibres.  Thus, since $g$ is a homotopy equivalence, the induced map
  $$g' \colon \Sigma F \to \Sigma D,$$ 
  between the homotopy fibres of the source and target of~$g$ as
  spaces over $K(G,1)$, where $\Sigma$ denotes here the (unreduced)
  suspension, is again a homotopy equivalence.  Therefore, $\pi_0(F)
  \cong [G:G']$ and each path-component of~$F$ must have trivial
  integral homology.  In addition, the homotopy fibre of $$q \colon
  \partial M \to K(G',1)$$ is identified with a path-component of~$F$,
  so the map~$q$ is acyclic (see Section~\ref{subsec:acyclicmaps}). In
  particular, $q$ is an integral homology equivalence and arises as
  the plus construction associated to the kernel of~$\iota$. As a
  consequence, if $\iota$ is injective, then $q \colon \partial M \to
  K(G',1)$ is a homotopy equivalence, so $\partial M$ is again
  aspherical. Conversely, it is well known that the double is
  aspherical if $M$ and $\partial M$ are aspherical and $\iota$ is
  injective.

  On the other hand, we do not know if the injectivity of~$\iota$ is
  necessary for the asphericity of the double.
\end{rem}

An interesting example of amenable glueings is given by connected
sums:

\begin{prop}[{\cite{vbc}\cite[Corollary~7.7]{Frigerio}}]\label{prop:connected:sum}
  Let $n \geq 1$ and let $M$ and $N$ be oriented closed connected
  $n$-manifolds. The following hold:
  \begin{enumerate}
  \item $\chi(M \# N) = \chi(M) + \chi(N) - \chi(S^{n})$;
  \item If $n \geq 3$, then $\sv{M \# N} = \sv{M} + \sv{N}$.
  \end{enumerate}
  In particular, if $n \geq 3$, we have:
  \begin{enumerate}
  \item[(3)] If $n$ is even and $\chi(M) = 0 = \chi(N)$, then $\chi(M
    \# N) \neq 0$;
  \item[(4)] If $\sv{M} = 0 = \sv{N}$, then $\sv{M \# N} = 0$.
  \end{enumerate}
\end{prop}

\begin{rem}\label{rem:connected:sum:not:asph}
  Note that the connected sum of aspherical manifolds in
  dimension~$\geq 3$ is never
  aspherical~\cite[Lemma~3.2]{lueck_aspherical}. Thus,
  Proposition~\ref{prop:connected:sum} cannot be used to produce
  counterexamples to Question~\ref{quest:gromov}.
\end{rem}

\begin{example}
  The formula in Proposition~\ref{prop:connected:sum}(2) fails in
  dimension~$2$. For example, hyperbolic surfaces have non-zero
  simplicial volume (Example~\ref{ex:pos:sim:vol}(1)) but the
  two-dimensional torus has zero simplicial volume
  (Example~\ref{ex:zero:simplicial:volume}.(1)).
\end{example}

\subsection{A relative version of Gromov's question}\label{subsec:relative}

We consider the following version of Question~\ref{quest:gromov} for manifolds
with boundary and show that it is a consequence of \pref{p:svchi} 
(Proposition~\ref{prop:gromov:implies:rel:gromov}).

\begin{quest}\label{quest:gromov:relative}
  Let $(M, \partial M)$ be an oriented compact aspherical
  manifold with non-empty $\pi_1$-injective aspherical boundary.
  Does the following implication hold?
  \begin{equation}
    \sv{M, \partial M} = 0  \Longrightarrow \chi(M, \partial M) = 0. 
    \plabel{SV$\chi$, $\partial$}{p:svchib'}
  \end{equation}
\end{quest}

\begin{rem}\label{rem:eulerrel}
  For every oriented compact connected $n$-manifold~$M$ with
  boundary~$\partial M$, we have $|\chi(M)| = |\chi(M, \partial M)|.$
  Indeed, when $n$ is even, we know that $\chi(\partial M) = 0$, so
  $\chi(M, \partial M) = \chi(M) - \chi(\partial M) = \chi(M).$ On the
  other hand, when $n$ is odd, we have $\chi(\partial M) = 2 \cdot
  \chi(M)$, so $\chi(M, \partial M) = \chi(M) - \chi(\partial M) =
  -\chi(M).$
\end{rem}

In order to disprove \pref{p:svchi} it suffices to find an example
that does not satisfy \pref{p:svchib}:

\begin{prop}\label{prop:gromov:implies:rel:gromov}
  Let $n \geq 1$. If all oriented closed aspherical
  manifolds of dimension~$n$ or~$n-1$ satisfy \pref{p:svchi}, then all
  oriented compact aspherical $n$-manifolds with non-empty
  $\pi_1$-injective aspherical boundary satisfy \pref{p:svchib}.
\end{prop}

\begin{proof}
  Let $(M,\partial M)$ be an oriented compact aspherical
  $n$-manifold with non-empty $\pi_1$-injective aspherical
  boundary. Then the double~$D(M) := M \cup_{\partial M} M$ is an
  oriented closed aspherical $n$-manifold
  (Remark~\ref{rem:double:aspherical}) and $\sv{D(M)} \leq 2 \cdot
  \sv{M,\partial M}$ (Example~\ref{exa:double:sv}).

  Suppose $\sv{M,\partial M} =0$. Then $\sv{D(M)} = 0$ and
  $\sv{\partial M} = 0$
  (Remark~\ref{rem:vanishing:rel:sv:implices:van:boundary}).  From
  \pref{p:svchi} in dimension~$n$ and $n-1$, respectively, we conclude
  \[ \chi\bigl(D(M)\bigr) = 0
  \qand
  \chi(\partial M) = 0.
  \]
  Therefore, we compute
  \begin{align*}
    \chi(M,\partial M)
    & = \chi(M) - \chi(\partial M)
    = \frac12 \cdot \bigl(\chi(D(M)) + \chi(\partial M)\bigr) - \chi(\partial M)
    = 0.
  \end{align*}
  Hence, $(M,\partial M)$ satisfies \pref{p:svchib}.
\end{proof}

\begin{example}\label{rem:svchibfails}
  Note that \pref{p:svchib} does not hold for all oriented compact
  aspherical manifolds without the $\pi_1$-injectivity
  condition on the boundary. For example, if we take an oriented
  closed connected hyperbolic even-dimensional manifold~$N$ and let $M
  := N \times D^2$, then $M$ and $\partial M = N \times S^1$ are
  aspherical and $\sv{M, \partial M} = 0$
  (Proposition~\ref{prop:sv:products} or
  Example~\ref{ex:zero:simplicial:volume}(4)).  On the other hand, we
  have $\chi(M, \partial M) = \chi(N) \cdot \chi(D^2, S^1) = \chi(N)
  \neq 0$.
\end{example}

\begin{rem}
  Note that by allowing the case of empty boundary,
  Proposition~\ref{prop:gromov:implies:rel:gromov} can be formulated
  as an equivalence between Question~\ref{quest:gromov:relative} and
  Question~\ref{quest:gromov}. The extension of Gromov's question to
  manifolds with boundary allows us in particular to explore
  Question~\ref{quest:gromov} by studying the properties of (vanishing
  of) the simplicial volume and the Euler characteristic along
  glueings of manifolds and compare their respective additivity
  properties. This viewpoint will also be explored in
  Section~\ref{sec:tqft}.
\end{rem}

\subsection{The comparison map}\label{subsec:comp}

Dually, simplicial volume can be expressed in terms of bounded
cohomology.  \emph{Bounded cohomology}
\[ H_b^*(\args;\R) := H^*(C_b^*(\args;\R))
\]
is the cohomology of the topological dual~$C_b^*(\args;\R)$ of the
singular chain complex, where the dual is taken with respect to the
$\ell^1$-norm. 
Bounded cohomology is then endowed with the $\ell^\infty$-seminorm,
denoted by $\| \cdot \|_\infty$.

The inclusion~$C_b^*(\args;\R) \hookrightarrow
C^*(\args;\R)$ induces a natural transformation
\[ \comp^* \colon H_b^*(\args;\R) \Longrightarrow H^*(\args;\R),
\]
which is called the \emph{comparison map}. A straightforward
application of the Hahn--Banach Theorem shows:

\begin{prop}[Duality principle~\cite{vbc}]\label{prop:duality}
  Let $(X,A)$ be a pair of spaces, let $k \in \N$, and let $\alpha\in
  H_k(X,A;\R)$.  Then
  \[ \|\alpha\|_1 = \sup \bigl\{\bigl| \langle \comp_{X,A}^k(\varphi), \alpha\rangle\bigr|
  \bigm| \varphi\in H^k_b(X,A;\R), \, \|\varphi\|_\infty \leq 1
  \bigr\}.
  \]
  In particular: If $(M, \partial M)$ is an oriented compact connected
  $n$-manifold with (possibly empty) boundary, then $\sv{M, \partial
    M}$ is the operator norm of the composition
  \[H^n_b(M, \partial M ; \R) \xrightarrow{\comp^n_{M, \partial M}}
    H^n(M, \partial M; \R) \stackrel{\cap [M, \partial M]}{\cong} \R
  \]
  and
  \[
  \comp^n_{(M, \partial M)} \mbox{ is surjective }
  \quad \Longleftrightarrow \quad
  \sv{M, \partial M} > 0.
  \]
\end{prop}

\begin{prop}\label{prop:cup}
  Let $M$ be an oriented closed connected $n$-manifold such that
  $\sv{M}=0$. Suppose that $x \in H^k(M; \R)$ is bounded
  (i.e. $x$ lies in the image of the comparison map $\comp^k_M$) 
   and let $x^*
  \in H^{n-k}(M; \R)$ be such that $x \cup x^* \neq 0$. Then $x^*$ is
  not bounded.
\end{prop}
\begin{proof}
  Let $y \in H^k_b(M; \R)$ be a class with~$\comp^k_M(y) =
  x$. \emph{Assume} for a contradiction that $x^* \in H^{n-k}(M;\R)$
  lies in the image of~$\comp^{n-k}_M$, that is, there is~$z \in
  H^{n-k}_b(M; \R)$ with~$\comp^{n-k}_M(z) = x^*$.
  The usual explicit formula for the cup-product on singular
  cohomology shows that the cup-product lifts to a cup-product
  on bounded cohomology. 
  Then
  \[ \comp^n_M(y \cup z) = \comp^k_M(y) \cup \comp^{n-k}_M(z) = x \cup x^* \neq 0,\]
  so $\comp^n_M$ maps surjectively onto $H^n(M; \R) \cong \R$. But
  this contradicts the assumption~$\sv{M} = 0$ according to
  Proposition~\ref{prop:duality}.
\end{proof}

The vanishing of the simplicial volume thus implies that not too many
classes can be bounded; dually, the vanishing of the simplicial volume
causes that there are many other classes with vanishing
$\ell^1$-semi-norm.

\begin{cor}
  Let $M$ be an oriented closed connected $n$-manifold satisfying~$\sv
  M = 0$ and let $N_*(M;\R) := \{ \alpha \in H_*(M;\R) \mid
  \|\alpha\|_1 = 0\}$. Then
  \[ \sum_{k \in \N} \dim_\R N_k(M;\R)
     \geq \frac12 \cdot \sum_{k \in \N} \dim_\R H_k(M;\R).
  \] 
\end{cor}
\begin{proof}
  On the one hand, by the duality principle
  (Proposition~\ref{prop:duality}), we have
  \[ \dim_\R H_k(M;\R) - \dim_\R N_k(M;\R) = \dim_\R (\im \comp^k_M)
  \]
  for all~$k \in \N$.  On the other hand, Poincar\'e duality and
  Proposition~\ref{prop:cup} imply that
  \[ \sum_{k\in \N} \dim_\R (\im \comp^k_M) \leq \frac12 \cdot \sum_{k \in \N} \dim_\R H^k(M;\R)
  = \frac12 \cdot\sum_{k\in \N} \dim_\R H_k(M;\R).
  \]
  Combining both estimates gives the claim.
\end{proof}

\subsection{Boundedness of the Euler class}\label{subsec:euler:class}

The Euler characteristic of an oriented closed connected smooth
$n$-manifold~$M$ can be expressed in terms of the Euler class~$e(M)
\in H^n(M;\R)$~\cite{Milnor-Stasheff} via
\[ \chi(M) = \bigl\langle e(M), [M] \bigr\rangle.
\]
The norm of the Euler class has been studied extensively in the
literature, especially, in connection with the existence of
\emph{flat} structures (a detailed account of results in this
direction is given in Frigerio's book~\cite{Frigerio}). The
boundedness of the Euler class is also closely related to
Question~\ref{quest:gromov}.

\begin{quest}\label{quest:euler:bounded}
  Let $M$ be an oriented closed aspherical smooth
  $n$-manifold.  Does the following property hold?
  \begin{equation}
    \text{The Euler class~$e(M) \in H^n(M;\R)$ is bounded.}
    \plabel{Eub}{p:euler:bounded}
  \end{equation}
\end{quest}

\begin{prop}\label{prop:bounded:euler:gromov:quest}
  Let $n \in \N$ and let $M$ be an oriented closed connected smooth
  $n$-manifold.  Then the following are equivalent:
  \begin{enumerate}
  \item The manifold~$M$ satisfies \pref{p:svchi}.
  \item The manifold~$M$ satisfies \pref{p:euler:bounded}.
  \end{enumerate}
\end{prop}
\begin{proof}
  Let $M$ satisfy \pref{p:svchi}. If $\sv{M} =0$, then $\langle e(M),
  [M]\rangle= \chi(M) = 0$. By duality, this implies that~$e(M) = 0$;
  in particular, $e(M)$ is bounded.  On the other hand, if $\sv{M} >
  0$, then the comparison map is surjective
  (Proposition~\ref{prop:duality}), hence $e(M)$ is also bounded. This
  shows that $M$ satisfies \pref{p:euler:bounded}.

  Conversely, suppose that $e(M)$ is bounded. Then
  \[ \bigl| \chi(M) \bigr| = \bigl| \bigl\langle e(M), [M] \bigr\rangle\bigr|
     \leq \bigl\| e(M) \bigr\|_\infty \cdot \sv M.
  \]
  As a consequence, if $\sv M = 0$, then $\chi(M) = 0$; i.e., $M$
  satisfies \pref{p:svchi}.
\end{proof}

\begin{rem}
  Let $(M, \partial M)$ be an oriented compact connected manifold with
  boundary. We may define
  $$e(M, \partial M) \in H^n(M, \partial M; \R)$$
  to be the Poincar\'e dual class to $\chi(M) \in \Z \cong H_0(M;
  \Z)$. We recall that $|\chi(M)| = |\chi(M, \partial M)|$
  (Remark~\ref{rem:eulerrel}). Then \pref{p:svchib} is equivalent to
  \begin{equation}
    e(M, \partial M) \in H^n(M, \partial M;\R) \text{ is bounded.}
    \plabel{Eub, $\partial$}{p:euler:boundedb}
  \end{equation} 
  The proof is the same as for
  Proposition~\ref{prop:bounded:euler:gromov:quest}.
\end{rem}

\begin{rem}
  It is well known that the Euler class of flat vector bundles is
  bounded~\cite[Section~13]{Frigerio}.  This shows that if $\sv{M} =
  0$ and the tangent bundle of~$M$ admits a flat connection, then
  $\chi(M) =
  0$~\cite{ivanov_turaev}\cite{bucher_finiteness}\cite[Theorem~13.11]{Frigerio}.

Moreover, it was conjectured~\cite[Conjecture~13.13]{Frigerio} that the Euler class of 
topologically flat sphere bundles admits a bounded representative. Monod and Nariman~\cite[Theorem~1.8]{monodnariman}
have recently proved that the Euler class of the (discrete) group of orientation-preserving
homeomorphisms of $S^3$ is unbounded.
\end{rem}

\begin{example}
  Assuming that Question~\ref{quest:gromov} has an affirmative answer,
  then Proposition~\ref{prop:bounded:euler:gromov:quest} has
  interesting implications for the existence of tangential maps
  between smooth manifolds. We recall that a map~$f \colon M \to N$
  between closed smooth manifolds is called \emph{tangential} if the
  vector bundles $TM$ and~$f^*TN$ are isomorphic.  As a consequence, a
  tangential map~$f \colon M \to N$ between oriented closed connected
  smooth manifolds preserves the Euler class up to sign.  Assuming
  Question~\ref{quest:gromov}, it follows that there
  cannot exist tangential maps $f \colon M \to N$ if $\chi(M)
  \neq 0$ and $N$ is aspherical with zero simplicial volume. Indeed, assuming 
  that $N$ satisfies \pref{p:svchi}, it follows that $e(N) = 0$ (since $\chi(N) = 0$). 
  Then, given a tangential map~$f \colon M \to N$, the classes $f^*(e(N))$ 
  and $e(M)$ agree up to sign, so $e(M) = 0$ (and therefore also $\chi(M) = 0$). 
\end{example} 

\section{Vanishing of the simplicial volume}\label{sec:van:sv}

In this section, we collect some known results on the simplicial
volume.  We will be mainly interested in describing sufficient
conditions for the vanishing of the simplicial volume. We also compare
those situations with the respective behaviour of the Euler
characteristic.

\subsection{Computations of the simplicial volume}\label{subsec:examples:van:sv}

In general, computing exact values of the simplicial volume is
difficult. For example, the problem of determining whether a given
(triangulated) manifold has vanishing simplicial volume or not is
undecidable~\cite[Chapter~2.6]{weinberger_computersrigidity}.  The two
major sources for (non-)vanishing results are amenability (which leads
to vanishing) and negative curvature (which leads to non-vanishing).

\begin{example}[Non-vanishing]\label{ex:pos:sim:vol}
  The following manifolds have positive simplicial volume:
  \begin{enumerate}
  \item Oriented closed connected hyperbolic manifolds~\cite{thurston,
    vbc};
  \item The compactification of oriented connected complete
    finite-volume hyperbolic manifolds~\cite{vbc, FM};
  \item Oriented closed connected manifolds with negative sectional
    curvature~\cite{inoueyano};
  \item Oriented closed connected locally symmetric spaces of
    non-compact type~\cite{bucherfirst, lafontschmidt}; 
  \item Oriented closed connected manifolds with non-positive
    sectional curvature and sufficiently negative intermediate
    Ricci curvature~\cite{connellwang};
  \item Oriented closed connected manifolds with non-positive
    sectional curvature and strong enough conditions at a single
    point~\cite{connellwang2};
  \item Oriented closed connected rationally essential (e.g., aspherical) 
    manifolds of dimension~$\geq 2$ with non-elementary hyperbolic fundamental group 
    (this follows via the duality principle from work of Mineyev
    on surjectivity of the comparison map~\cite{mineyev});
  \item Oriented closed connected rationally essential manifolds of 
    dimension~$\geq 2$ with non-elementary relatively hyperbolic
    fundamental group~\cite{BHrel};
  \item
    Non-vanishing of simplicial volume is inherited through
    a proportionality principle~\cite{vbc,thurston,loeh_pp}: If $M$
    and $N$ are oriented closed connected Riemannian manifolds with
    isometric universal coverings, then $\sv M > 0$ if and only if
    $\sv N > 0$.
  \end{enumerate}
\end{example}

In this paper, we focus our attention on vanishing results for the
simplicial volume.  The following example contains some known 
vanishing results:

\begin{example}[Vanishing]\label{ex:zero:simplicial:volume}
  The following manifolds have zero simplicial volume:
  \begin{enumerate}
  \item Oriented closed connected $n$-manifolds with amenable
    fundamental group and $n>0$~\cite{vbc} or, more
    generally, with $n$-boundedly acyclic fundamental group~\cite{BAc,
      vbc}; finitely presented non-amenable boundedly acyclic groups
    have been recently constructed~\cite{FLMbdd, monodthompson};
  \item Oriented compact connected $n$-manifolds~$M$ with non-empty
    boundary such that both the fundamental groups of $M$ and
    of~$\partial M$ are amenable~\cite{vbc};
  \item More generally, oriented compact connected $n$-manifolds~$M$
    with non-empty boundary such that $\sv{\partial M} = 0$ and the
    connecting homomorphism $H^{n-1}_b(\partial M) \to H^n_b(M, \partial M)$ is
    surjective. Manifolds satisfying the latter condition can be
    constructed by taking manifolds whose boundary inclusion is
    $\pi_1$-surjective and such that their fundamental group lies
    in~$\lex$~\cite{Bouarich}. Recall that $\lex$ groups are those
    groups~$\Gamma$ such that every epimorphism~$\Lambda
    \twoheadrightarrow \Gamma$ induces an injective map in bounded
    cohomology in every degree.  Examples of $\lex$ groups contain
    free groups, amenable groups~\cite{Bouarich}, boundedly acyclic
    groups and certain extensions of these~\cite[Remark~3.8]{FLMbdd}.
  \item Oriented compact connected manifolds with (possibly empty)
    boundary that admit a self-map~$f$ of degree~$\deg(f) \not\in \{0,
    1, -1\}$~\cite{vbc}; 
  \item Oriented closed connected manifolds that are the boundary of an
    oriented compact connected manifolds with zero simplicial volume
    (Remark~\ref{rem:vanishing:rel:sv:implices:van:boundary});
  \item Oriented closed connected $n$-manifolds that admit a smooth
    non-trivial $S^1$-action~\cite{yano}.  More generally, manifolds
    admitting an $F$-structure also have zero simplicial
    volume~\cite{cheegergromov, paternainpetean};
  \item Oriented closed aspherical manifolds supporting an
    affine structure whose holonomy map is injective and contains a
    pure translation~\cite{BCL};
  \item Oriented closed connected smooth manifolds with zero minimal
    volume~\cite{vbc, BCGminimal} or zero minimal volume
    entropy~\cite[p.~37]{vbc}\cite{balacheffkaram};
  \item Oriented closed connected graph $3$-manifolds~\cite{soma,vbc};
  \item All mapping tori of oriented closed connected
    $3$-manifolds~\cite{bucherneofytidis}; however, the general
    behaviour of simplicial volume of general mapping tori is very
    diverse~\cite{kastenholzreinhold}.
\end{enumerate}
\end{example}

\begin{rem}\label{rem:eulerselfmap}
  In view of Question~\ref{quest:gromov}, it would be interesting to
  understand whether all oriented closed aspherical
  manifolds that admit a self-map~$f$ of degree~$\deg(f) \not\in \{0,
  1, -1\}$ must have zero Euler characteristic. For surfaces this is
  clearly the case (the only candidate being the torus).  More
  generally, this is known to be true whenever the fundamental group
  of the aspherical manifold is Hopfian~\cite{agol_MO}.
  The statement was claimed in full generality by
  Sullivan~\cite[Footnote~23, p.~318]{Sullivan} without a proof. 
\end{rem}

\subsection{Amenable covers: The closed case}\label{subsec:amen:cover:closed}

A useful approach to investigate the vanishing of simplicial volume is
to consider \emph{amenable covers}.  This idea dates back to
Gromov~\cite{vbc} and it was then developed further by many
authors~\cite{Ivanov, GGW, FM:Grom, Frigerio:ell1, LS,
  Ivanov_bac_covers, CLM, bsfibre, lmfibre, georgevanishing}. We use
the terminology of amenable category~\cite{GGW, CLM, lmfibre}:

\begin{defi}[Amenable covers and category]\label{defi:amenable:covers:category}
  \hfil
  \begin{enumerate}
  \item Let $X$ be a topological space and let $U$ be a subset
    of~$X$. We say that $U$ is \emph{amenable} in~$X$ if for every~$x \in
    \, U$ the image
    $$
    \im\bigl(\pi_1(U \hookrightarrow X, x)\bigr) \leq \pi_1(X, x)
    $$
    is amenable. We say that an open cover of~$X$ is \emph{amenable}, if
    it consists of amenable sets.
  \item The \emph{amenable category} of~$X$, denoted by~$\amcat(X)$,
    is the minimal integer~$n$ such that $X$ admits an open amenable
    cover with cardinality~$n$. If no such integer exists, we simply
    set $\amcat(X) = +\infty$.
\end{enumerate}
\end{defi}

The vanishing results for open amenable covers are usually stated in
terms of assumptions on the multiplicity of the cover instead of the
cardinality. These assumptions essentially are the same when working
with paracompact Hausdorff spaces~\cite[Remark~3.13]{CLM}.

The importance of amenable covers in our setting is demonstrated by
the following two results:

\begin{thm}[Gromov's vanishing theorem~{\cite[p.~40]{vbc}}]\label{thm:grom:vanishing:thm}
  Let $M$ be an oriented closed connected $n$-manifold. Then,
  \[ \amcat(M) \leq n \Longrightarrow \sv M = 0.
  \]
\end{thm}

A similar result for the Euler characteristic has been proved by
Sauer~\cite{sauerminvol}:

\begin{thm}[Euler characteristic and amenable covers]\label{thm:van:thm:euler:characteristic}
  Let $M$ be an oriented closed aspherical
  $n$-manifold. Then,
  \[ \amcat(M) \leq n \Longrightarrow \chi(M) = 0.
  \]
\end{thm}

\begin{rem}\label{rem:euler:non:asph}
  In Theorem~\ref{thm:van:thm:euler:characteristic}, the asphericity
  assumption is crucial: every even-dimensional sphere provides a
  counterexample in the non-aspherical setting.
\end{rem}

In particular, Theorems~\ref{thm:grom:vanishing:thm}
and~\ref{thm:van:thm:euler:characteristic} show that all oriented
closed aspherical manifolds~$M$ with~$\amcat(M) \leq
\dim(M)$ satisfy \pref{p:svchi}.

\begin{example}
  Certain fibre bundles yield examples of the situation arising in
  Theorems~\ref{thm:grom:vanishing:thm}
  and~\ref{thm:van:thm:euler:characteristic}. Let $N \hookrightarrow M
  \to B$ be a fibre bundle of oriented closed connected manifolds and
  suppose that
  \[ \amcat(N) \leq \frac{\dim(M)}{\dim(B)+1}.
  \]
 Then, we have $\amcat(M) \leq
 \dim(M)$~\cite[Corollary~1.2]{lmfibre} (and $\dim(M) \geq \dim(B) + 1 \geq 1$). Using
 Theorems~\ref{thm:grom:vanishing:thm}
 and~\ref{thm:van:thm:euler:characteristic}, we conclude:
  \begin{itemize}
  \item[(a)] $\sv{M} = 0$ (Theorem~\ref{thm:grom:vanishing:thm}).
  \item[(b)] If $M$ is aspherical, then also~$\chi(M) = 0$
    (Theorem~\ref{thm:van:thm:euler:characteristic}).
  \end{itemize}
  Concerning (b), it should be noted that in this case the asphericity
  of~$N$ is also sufficient in order to conclude that~$\chi(M) =
  0$. Indeed the hypothesis on~$\amcat(N)$ shows that
  \[ \amcat(N) \leq \frac{\dim(M)}{\dim(B) +1}
       = \frac{\dim(B) + \dim(N)}{\dim(B) +1}
       =\frac{\dim(B)}{\dim(B) + 1} + \frac{\dim(N)}{\dim(B) +1}.
  \]
  Since $\amcat(N)$ is an integer, it follows that $\amcat(N) \leq
  \dim(N)$.  Thus, $\chi(N) = 0$
  (Theorem~\ref{thm:van:thm:euler:characteristic}) and so $\chi(M) =
  \chi(N) \cdot \chi(B) = 0$.
\end{example}

Recently, Gromov's vanishing theorem was extended to weaker
situations, such as weakly boundedly acyclic open covers and other
more general homotopy colimit
decompositions~\cite{Ivanov_bac_covers,georgevanishing}. This new
context suggests in particular the following question:

\begin{quest}\label{quest:weakly:boundedly:acyclic:cover:euler}
  Let $M$ be an oriented closed aspherical
  $n$-manifold. Assuming there exists a weakly boundedly acyclic open
  cover of~$M$ in the sense of Ivanov~\cite{Ivanov_bac_covers} with
  cardinality at most~$n$, does it follow that $\chi(M) = 0$\;?
\end{quest}

A negative answer to
Question~\ref{quest:weakly:boundedly:acyclic:cover:euler} would also
produce closed aspherical examples that do not satisfy \pref{p:svchi}.
In fact, no example of an oriented closed aspherical
manifold~$M$ with vanishing simplicial volume and~$\amcat(M) = \dim(M)
+ 1$ seems to be known.

\subsection{Amenable covers: The case with boundary}\label{subsec:amen:cover:relative}

The vanishing theorem can also be extended to the relative setting via
Gromov's vanishing-finiteness
theorem~\cite{vbc,FM:Grom,Ivanov_open_covers}. We quickly recall the
formulation of the vanishing-finiteness theorem and mention two
convenient special cases for compact manifolds with boundary: the
relative vanishing theorem and the case of locally co-amenable
subcomplexes.

For an oriented connected (possibly non-compact) manifold~$M$ without
boundary, the \emph{locally finite simplicial volume} is defined by
\[ \svlf{M} := \inf \bigl\{ |c|_1
\bigm| \text{$c \in C_n^{\lfop}(M;\R)$ is a fundamental cycle of~$M$}\bigr\}, 
\]
where $C_*^{\lfop}(\args;\R)$ denotes the (singular) locally finite
chain complex of~$M$. Because locally finite chains are not
necessarily~$\ell^1$, the locally finite simplicial volume is not
always finite.

\begin{thm}[Vanishing-finiteness theorem~\protect{\cite[p.~58]{vbc}\cite[Section~7.2]{FM:Grom}\cite[Theorem~4.6]{Ivanov_open_covers}}]\label{thm:vanfin}
  Let $M$ be an oriented connected $n$-manifold without boundary and
  let $(U_i)_{i \in \N}$ be an open cover of~$M$ with the following
  properties:
  \begin{itemize}
  \item[(i)] For each~$i \in \N$, the subset~$U_i \subset M$ is
    amenable and relatively compact.
  \item[(ii)] The sequence~$(U_i)_{i \in \N}$ is \emph{amenable at
    infinity}, i.e., there exists an increasing sequence~$(K_i)_{i \in
    \N}$ of compact subsets of~$M$ such that:
    \begin{itemize}
    \item[(a)] The family~$(M \setminus K_i)_{i \in \N}$
    is locally finite;
    
    \item[(b)] For all~$i \in
    \N$, we have $U_i \subset M \setminus K_i$;
    
    \item[(c)] For all sufficiently large~$i \in \N$, the
    set~$U_i$ is amenable in~$M \setminus K_i$.
    \end{itemize}    
  \item[(iii)] 
    The multiplicity of~$(U_i)_{i \in \N}$ is at most~$n$.
  \end{itemize}
  Then $\svlf M = 0$.
\end{thm}

\begin{rem}
  Properties (a) and (b) in (ii) show that every sequence $(U_i)_{i \in \N}$
  which is amenable at infinity is necessarily locally finite.
\end{rem}

Vanishing of the locally finite simplicial volume leads to vanishing
results for the relative simplicial volume:

\begin{rem}\label{rem:lfrel}
  If $(M,\partial M)$ is an oriented compact connected manifold, then
  $\sv{M,\partial M} \leq
  \svlf{\interior(M)}$~\cite{vbc}\cite[Proposition~5.12]{Loehthesis}. In
  general, this inequality is
  strict~\cite[p.~10]{vbc}\cite[Example~6.17]{Loehthesis}:
  E.g.,~$\sv{[0,1], \{0,1\}} < \infty = \svlf{(0,1)}$.
  
  On the other hand,  there is no known example for which $\svlf{\interior(M)}$ is finite and distinct from $\|M,\partial M\|$.
\end{rem}

\begin{thm}[Relative vanishing theorem]\label{thm:relvan}
  Let $(M, \partial M)$ be an oriented compact connected
  $n$-manifold that admits an amenable open cover~$(U_i)_{i \in I}$
  with the following properties:
  \begin{itemize}
  \item[(i)] The multiplicity of~$(U_i)_{i \in I}$ is at most~$n$.
  \item[(ii)] The multiplicity of~$(U_i \cap \partial M)_{i \in I}$ is at
    most~$n-1$.
  \item[(iii)] For each~$i \in I$, the set~$U_i \cap \partial M$ is
    amenable in~$\partial M$.
  \end{itemize}
  Then $\sv{M,\partial M} = 0$.
\end{thm}

\begin{proof}
  In view of Remark~\ref{rem:lfrel}, it suffices to show that the
  vanishing-finiteness theorem (Theorem~\ref{thm:vanfin}) can be
  applied to~$\interior(M)$.  To this end, we modify the given open
  cover~$(U_i)_{i \in I}$ as follows. Up to homeomorphism, we can
  write
  \[
  \interior (M) \cong M \cup_{\partial M} \bigl( \partial M \times [0, +\infty) \bigr).
  \]
  We extend~$(U_i)_{i \in I}$ to the right-hand side by replacing the
  sets~$V$ of~$(U_i)_{i \in I}$ intersecting~$\partial M$ with~$V \cup
  ((V \cap \partial M) \times [0,\infty))$. Let $(V_i)_{i \in I}$
  be the resulting open cover of~$\interior(M)$. We now upgrade
  this cover~$(V_i)_{i \in I}$ to a locally finite cover made of
  relatively compact sets without increasing the multiplicity;
  because the intersection of~$(U_i)_{i \in I}$ with~$\partial M$
  has multiplicity at most~$n-1$, this is indeed possible by
  a standard procedure~\cite[Proof of Theorem~11.2.3, p.~144]{FM:Grom}.
  The resulting open cover satisfies all the conditions required
  by the vanishing-finiteness theorem (Theorem~\ref{thm:vanfin}).
\end{proof}

Based on Theorem~\ref{thm:relvan}, the following question is a special
case of Question~\ref{quest:gromov:relative}:

\begin{quest}
  Let $(M, \partial M)$ be an oriented compact aspherical
  $n$-manifold with non-empty $\pi_1$-injective aspherical boundary
  that admits an open cover as in Theorem~\ref{thm:relvan}. Then, do
  we have $\chi(M, \partial M) = 0$\;?
\end{quest}

Amenable covers as in the vanishing-finiteness theorem
(Theorem~\ref{thm:vanfin}) appear naturally in the presence of locally
co-amenable subcomplexes~\cite{vbc, FM:Grom}:

\begin{defi}[Locally co-amenable subcomplex~{\cite[p.~59]{vbc}\cite[Definition~11.2.1]{FM:Grom}}]
  Let $M$ be an oriented compact connected PL-manifold with non-empty
  boundary and let $P$ be a simplicial complex such that $M \cong
  |P|$.  Assume that there exists a simplicial complex $K \subset P$
  such that $|K| \subset \interior(M)$ and $M$ is homeomorphic to a
  closed regular neighbourhood of $K$ inside
  $P$~\cite[Definition~11.1.4]{FM:Grom}.  Suppose also that $K$ has
  codimension at least~$2$ in~$M$. Then, $K$ is called \emph{locally
    co-amenable in~$P$ (or in $M$)} if for each vertex $v \in \, (K'')^0$ of the
  second barycentric subdivision~$K''$ of~$K$ we have that
  $$
  \pi_1(|S_v \setminus (S_v \cap K)|)
  $$
  is amenable. Here, $S_v$ denotes the simplicial sphere in~$P''$
  centered at~$v$.
\end{defi}

\begin{rem}\label{rem:loccoam:homotopy:equiv}
  If $K$ is locally co-amenable in~$M$, then $M$ is homotopy
  equivalent to $K$.
\end{rem}

\begin{rem}
  If both $M$ and~$\partial M$ are aspherical and $M$ admits a locally
  co-amenable subcomplex~$K$, then the boundary inclusion is
  \emph{not} $\pi_1$-injective: Indeed, if $\partial M \hookrightarrow
  M$ were $\pi_1$-injective, then $\pi_1(\partial M)$ would be
  isomorphic to a subgroup of~$\pi_1(M)$. So, asphericity and the
  Shapiro lemma show that
  \[ \cd \pi_1(M) \geq \cd \pi_1(\partial M) = \dim(M) - 1.
  \]
  However, as $M$ is homotopy equivalent to~$K$
  (Remark~\ref{rem:loccoam:homotopy:equiv}), also $K$ is aspherical
  and thus
  \[ \cd \pi_1(M) \leq \dim(K) \leq \dim(M) - 2,
  \]
  which is a contradiction.
\end{rem}

\begin{prop}\label{prop:loc:coam:relative}
  If $M$ is an oriented compact connected PL-manifold with non-empty
  boundary that admits a locally co-amenable subcomplex, then $\sv{M,
    \partial M} = 0$.
\end{prop}
\begin{proof}
  Under the given assumptions, the vanishing-finiteness theorem
  (Theorem~\ref{thm:vanfin}) applies
  to~$\interior(M)$~\cite[Theorem~11.2.3]{FM:Grom}.  Therefore,
  $\sv{M,\partial M} = 0$ (Remark~\ref{rem:lfrel}).
\end{proof}

\subsection{Products of manifolds}\label{subsec:products:sv}

While the Euler characteristic is multiplicative with respect
to products, the product behaviour for the simplicial volume
is more delicate. If one of the factors is closed, the vanishing
behaviour of simplicial volume is controlled by the factors:

\begin{prop}[Simplicial volume and products~{\cite{vbc, Bucher:product}\cite[Proposition~C.7]{Loehthesis}}]\label{prop:sv:products}
  Let $M$ be an oriented closed connected $m$-manifold and let $N$ be
  an oriented compact connected $n$-manifold with (possibly empty)
  boundary. Then, we have
  $$
  \sv{M} \cdot \sv{N, \partial N} \leq \sv{M \times N, \partial (M \times N)} \leq {{n+m}\choose{m}} \cdot \sv{M} \cdot \sv{N, \partial N}.
  $$
\end{prop}

The exact values in general are unknown; the only known non-zero
computation is the product of two closed surfaces~\cite{bucherprod}.

On the other hand, the product of at least three compact manifolds
with non-empty boundary always has vanishing simplicial volume:

\begin{prop}\label{prop:vanishing:triple:product}
  Let $M_1, M_2, M_3$ be oriented compact connected PL-manifolds with
  non-empty boundary. Then, we have
  $$
  \sv{M_1 \times M_2 \times M_3, \partial (M_1 \times M_2 \times M_3)} = 0 \ .
  $$
\end{prop}
\begin{proof}
  In this situation, $M_1 \times M_2 \times M_3$ admits a locally
  co-amenable subcomplex~\cite[Example~(a),
    p.~59]{vbc}\cite[Theorem~14]{FM:Grom}. Therefore, we can apply
  Proposition~\ref{prop:loc:coam:relative}.

An alternative proof is as follows: Let $n = \dim(M_1 \times M_2 \times M_3)$.
The homotopy fibre of the boundary inclusion
in this case has trivial fundamental group and thus has trivial bounded cohomology.
Hence the induced map in bounded cohomology 
$H_b^k(M_1 \times M_2 \times M_3) \to H_b^k(\partial (M_1 \times M_2 \times M_3))$
is an isomorphism in every degree~\cite{vbc, BAc}.
The long exact sequence of the pair then implies that 
$H_b^n(M_1 \times M_2 \times M_3, \partial (M_1 \times M_2 \times M_3))$
is trivial, whence $\sv{M_1 \times M_2 \times M_3, \partial (M_1 \times M_2 \times M_3)} = 0$
(Proposition~\ref{prop:duality}).
\end{proof}

A nice application of the previous result is the following:

\begin{example}\label{example:triple:product:tori}
  Let $(T^2)^{\times}$ denote a $2$-dimensional torus with an open
  disk removed.  Then the simplicial volume of the product of (at
  least) three copies of $(T^2)^{\times}$ vanishes. On the other
  hand, $\chi((T^2)^{\times} \times (T^2)^{\times} \times
  (T^2)^{\times}) = -1$.
\end{example}

While triple products have zero relative simplicial volume, the
situation remains undecided in the following case of products of two
compact manifolds~\cite[Question~8]{FM:Grom}:

\begin{quest}\label{quest:vanishing:product:2:simvol}
  Let $M$ and $N$ be oriented compact connected manifolds with
  non-empty connected boundaries. Does it follow that
  \begin{equation}
    \sv{M \times N, \partial (M \times N)} = 0 \;\text{?}
    \plabel{SV$\times$}{p:svprod}
  \end{equation}
\end{quest}

\begin{rem}
  Without the connectedness assumption on the boundary, there are
  products that do not satisfy \pref{p:svprod}.  For example, it is
  known that the product of a compact hyperbolic surface with boundary
  and a closed interval has non-zero relative simplicial
  volume~\cite[p. 17]{vbc} \cite[Corollary~6.2]{Loehthesis}.
\end{rem}

The following proposition shows an interesting connection between
\pref{p:svprod} and Question~\ref{quest:gromov}.

\begin{prop}\label{prop:gromov:quest:products}
  Let $M$ and $N$ be oriented compact aspherical manifolds
  with non-empty $\pi_1$-injective aspherical boundary that satisfy
  $\chi(M) \cdot \chi(N) \neq 0$. Furthermore, suppose that $M$ and
  $N$ have dimensions of different parity and satisfy
  \pref{p:svprod}. Then $\partial(M \times N)$ does not satisfy
  \pref{p:svchi}.
\end{prop}

\begin{proof}
  As $M$ and $N$ satisfy \pref{p:svprod}, we have 
  $
  \sv{M \times N, \partial (M \times N)} = 0. 
  $
  In particular, $\sv{\partial (M \times N)} = 0$
  (Remark~\ref{rem:vanishing:rel:sv:implices:van:boundary}).
  Moreover, the boundary $\partial(M \times N) = (\partial M \times N)
  \cup_{\partial M \times \partial N} (M \times \partial N)$ is
  aspherical and has even dimension. So it suffices to show that
  $\partial (M \times N)$ has non-zero Euler characteristic. Since
  $\chi(M \times N) = \chi(M) \cdot \chi(N) \neq 0$ and $M \times N$
  is odd-dimensional, we have $\chi(\partial (M \times N)) = 2 \cdot
  \chi(M \times N) \neq 0$.
\end{proof}

\subsection{Small aspherical fillings}\label{subsec:filling:sv}

We now come to a higher order version of vanishing, which asks for
``small'' aspherical fillings of the given manifold with vanishing
Euler characteristic/simplicial volume.  We will be mainly interested
in filling aspherical $3$-manifolds with amenable fundamental group.

\begin{defi}[{\cite[p.~3]{Edmonds}}]
  Let $M$ be an oriented closed aspherical $3$-manifold. We
  define
  \[
  \Fillchi (M) \coloneqq \min_{W \in F(M)} |\chi(W)| \ ,
  \]
  where $F(M)$ denotes the class of all oriented compact aspherical $4$-manifolds $W$ with $\pi_1$-injective boundary~$M$.
\end{defi}

\begin{quest}[Edmonds~{\cite[p.~3]{Edmonds}}]\label{quest:edmonds}
  Does there exist an oriented closed connected $3$-manifold~$M$ with
  $\Fillchi(M) \neq 0$?
\end{quest}

In the same spirit, we could ask the corresponding question for the
simplicial volume:

\begin{defi}
  Let $M$ be an oriented closed connected $3$-manifold. We say that
  $M$ \emph{admits a small aspherical filling} if there exists~$W \in
  F(M)$ such that~$\sv{W,M} = 0.$
\end{defi}

The previous definition suggests the following question:

\begin{quest}\label{quest:sv:edmonds}
  Does every oriented closed aspherical $3$-manifold satisfy
  the following implication?
  \begin{equation}
    \pi_1(M) \mbox{ is amenable } \Longrightarrow M \mbox{ admits a
      small aspherical filling.}  \plabel{Fill}{p:fillsv}
  \end{equation}
\end{quest}

Question~\ref{quest:sv:edmonds} can be interpreted as a manifold
variant of the \emph{uniform boundary condition}
($\ubc$)~\cite{MM}. Recall that a space~$X$ satisfies~$\ubc$ in
dimension~$n$ if there exists a constant~$K > 0$ such that every
boundary~$c \in \, \im \partial_{n+1} \subset C_n(X;\R)$ can be filled
$K$-efficiently, i.e., there exists a chain~$b \in \, C_{n+1}(X;\R)$
such that $\partial_{n+1} b = c$ and $|b|_1 \leq K \cdot |c|_1$.
Spaces with amenable fundamental groups satisfy~$\ubc$ in all
dimensions~\cite{MM}.  Therefore, in a similar way,
Question~\ref{quest:sv:edmonds} asks whether the small fundamental
cycles of oriented closed connected $3$-manifolds~$M$ with amenable
fundamental group can be filled efficiently using relative fundamental
cycles of $4$-manifolds with $M$ as $\pi_1$-injective boundary.

Similar quantified bordism problems have been successfully studied
in more geometric contexts~\cite{CDMW}.

\begin{rem}
  \pref{p:fillsv} does \emph{not} hold in dimension~$1$. Indeed, the
  only surfaces that have the circle as $\pi_1$-injective boundary are
  hyperbolic surfaces with totally geodesic boundary. All these have
  uniformly positive simplicial volume
  (Example~\ref{ex:pos:sim:vol}(2)).

  \pref{p:fillsv} holds in dimension~$2$. The only candidate to check
  is the $2$-torus, which is the $\pi_1$-injective boundary of~$S^1
  \times (T^2)^\times$.
\end{rem}

Our interest in Question~\ref{quest:sv:edmonds} is motivated by the
following open problem in $4$-dimensional topology:

\begin{conj}[{\cite[Conjecture~1]{Edmonds}}]\label{conj:chi:1}
  There exists an oriented closed aspherical
  $4$-manifold~$M$ with~$\chi(M) = 1$.
\end{conj}

The last conjecture and Question~\ref{quest:gromov} are connected as
follows:

\begin{prop}\label{lemma:conj:chi:1:vs:Gromov}
 Suppose that the following hold:
  \begin{itemize}
  \item[(a)] Every oriented closed aspherical $3$-manifold
   satisfies \pref{p:fillsv};
  \item[(b)] All oriented closed aspherical $4$-manifolds
    satisfy \pref{p:svchi}.
  \end{itemize}  
  Then Conjecture~\ref{conj:chi:1} holds.
\end{prop}
\begin{proof}
  Edmonds \cite{Edmonds} constructed an oriented compact
  aspherical $4$-manifold $W$ with non-empty $\pi_1$-injective
  aspherical boundary and $\chi(W) =
  1$~\cite[Proposition~4.1]{Edmonds}.  Moreover, $\partial W$ is a
  torus bundle over the circle~\cite[Proposition~4.1]{Edmonds}. This
  shows that $\partial W$ has amenable fundamental group.
  
  Using \pref{p:fillsv}, there exists an oriented compact aspherical $4$-manifold~$W'$ with $\pi_1$-injective
  boundary~$\partial W' \cong \partial W$ (orientation-reversing) and $\sv{W', \partial W'} =
  0$.  Moreover, by hypothesis, \pref{p:svchi} is satisfied both for
  all $3$- and $4$-dimensional oriented closed aspherical
  manifolds (\pref{p:svchi} is automatically satisfied in odd
  dimensions).  Hence, Proposition~\ref{prop:gromov:implies:rel:gromov}
  shows that $(W',\partial W')$ satisfies \pref{p:svchib}, and so
  $\chi(W', \partial W') = 0$. Therefore,
  $$
  M := W \cup_{\partial W \cong \partial W'} W'
  $$
  is an oriented closed aspherical $4$-manifold with
  $$
  \chi(M) = \chi(W) + \chi(W', \partial W') = 1 + 0 = 1 \ .
  $$
  Therefore $M$ provides the required example for
  Conjecture~\ref{conj:chi:1}.
\end{proof}

\section{Simplicial volume and cobordism categories}\label{sec:tqft}

In this section, we will introduce the amenable cobordism category and
explain how the simplicial volume extends to a symmetric monoidal
functor on this cobordism category. In other words, the simplicial
volume defines an invertible TQFT in this restricted
sense. Interestingly, it will be shown that the simplicial volume does
not extend to a functor on the whole cobordism category of smooth
oriented manifolds. This fact reflects the (non-)additivity properties
of the simplicial volume.

Viewing the simplicial volume as an invertible TQFT will allow us to
obtain some interesting information about the fundamental group of the
amenable cobordism category and its variations. Specifically, we will
show that this fundamental group is \emph{not} finitely generated
(Theorem~\ref{thm:cob}). This result is based on the following
computations of simplicial volume in dimension~$4$:

\begin{rem}\label{rem:sv4nonfin}
  For~$n \in \N$, let $\SV(n) \subset \R_{\geq 0}$ denote the set of
  simplicial volumes of all oriented closed connected $n$-manifolds.
  Then $\SV(n)$ is a countable submonoid of~$(\R_{\geq 0},
  +)$~\cite[Remark~2.3]{heuerloeh}.
  
  If $n \geq 4$, then $\SV(n)$ has no gap at zero~\cite{heuerloeh} and
  thus is non-discrete. Moreover, $\SV(3)$ contains the set of all
  volumes of oriented closed connected hyperbolic $3$-manifolds
  (scaled by~$1/v_3$) and thus is non-discrete~\cite{thurston}.
  Therefore, if $n \geq 3$, then the additive monoid~$\SV(n)$ is
  \emph{not} finitely generated.

  Moreover, $\SV(4)$ contains an infinite family of values that are
  linearly independent over~$\Q$~\cite{heuerloeh_trans}.
\end{rem}

We first consider the simpler case of the connected sum monoid and
prove that it is not finitely generated
(Section~\ref{subsec:monoid}). In Section~\ref{subsec:tqft}, we
explain how to view the simplicial volume as a symmetric monoidal
functor on the amenable cobordism category and use this description to
deduce the non-finite generation of the fundamental group of the $4$-dimensional 
amenable cobordism category (Theorem~\ref{thm:cob}). Finally, we prove
that the functor of simplicial volume cannot be extended to an
invertible TQFT on the whole cobordism category
(Proposition~\ref{prop:tqft}).

\subsection{The connected sum monoid}\label{subsec:monoid}

For $n \in \N$, let $\Mon n$ denote the monoid, whose elements are
diffeomorphism classes of oriented closed connected smooth
$n$-manifolds and whose operation is given by connected sum.  By the
classification of surfaces, the monoid~$\Mon 2$ is generated by the
$2$-torus. This finite generation fails in higher dimensions:

\begin{prop}
  Let $n \in \N_{\geq 3}$. Then the monoid~$\Mon n$ is \emph{not}
  finitely generated.
\end{prop}
\begin{proof}
  As simplicial volume is additive in dimension~$\geq 3$ with respect
  to connected sums (Proposition~\ref{prop:connected:sum}), we can
  view the simplicial volume as a monoid homomorphism
  \[ S \colon \Mon n \longrightarrow \R_{\geq 0} 
  \]
  from~$\Mon n$ to the additive monoid~$(\R_{\geq 0},+)$. The
  submonoid~$S(\Mon n)$ is not finitely generated
  (Remark~\ref{rem:sv4nonfin}).  Because finite generation is
  preserved by monoid homomorphisms, we conclude that $\Mon n$ is
  \emph{not} finitely generated.
\end{proof}

\begin{rem}
As suggested by the referee, the previous result also admits a more geometric 
proof: In every dimension~$n \geq 3$ there exist \emph{infinitely many} hyperbolic
$n$-manifolds (and none of them is a non-trivial connected sum).
\end{rem}

\subsection{Simplicial volume as a TQFT}\label{subsec:tqft}

The simplicial volume can be viewed as an (invertible)~TQFT defined on
an appropriate cobordism category of oriented smooth manifolds. This
is essentially a basic consequence of known additivity properties of
the simplicial volume. For background material about cobordism
categories and TQFTs, we refer the interested reader to the work of
Abrams~\cite{Abrams} and the book by Kock~\cite{Kock}, both of which
focus especially on the $2$-dimensional case, and to the lecture notes
of Debray, Galatius, and Palmer~\cite{DGPinvfield}, which contain an
excellent exposition of the classification of invertible TQFTs
following major recent developments in the field.

\subsubsection{Cobordism categories.} 
For~$d\in \N$, let $\Cob_d$ denote the $d$-dimensional (discrete)
cobordism category of oriented manifolds~\cite{Kock, DGPinvfield}.
The objects of~$\Cob_d$ are oriented closed smooth
$(d-1)$-manifolds~$M$, one from each diffeomorphism class. A morphism
from $M$ to~$N$ in~$\Cob_d$ is an equivalence class of $d$-dimensional
oriented smooth cobordisms~$(W; \partial_{\text{in}}W,
\partial_{\text{out}} W)$ equipped with orientation-preserving diffeomorphisms~$-M
\stackrel{\cong}{\to} \partial_{\text{in}} W$ (incoming boundary) and~$N \stackrel{\cong}{\to} \partial_{\text{out}} W$ (outgoing boundary). The
equivalence relation is given by orientation-preserving
diffeomorphisms that preserve the boundary pointwise.  Composition of
morphisms in~$\Cob_d$ is given by glueing of cobordisms, using the
given identifications of the boundary components. The
category~$\Cob_d$ is a symmetric monoidal category under the operation
of disjoint union.


\subsubsection{Amenability conditions.}
Let $G$ be a class of groups that is closed under isomorphisms. We
consider the subcategory~$\Gcob d \subset \Cob_d$ defined as follows:
The objects are those manifolds with fundamental group in~$G$ (for
each component). The morphisms are the cobordisms~$(W; M, N)$ such
that $M \hookrightarrow W$ and $N \hookrightarrow W$ are
$\pi_1$-injective (for all components). It should be noted that~$\Gcob
d$ is indeed a subcategory of~$\Cob_d$, i.e., that $\Gcob d$ is closed
under composition. To see this, we only need to check that the
$\pi_1$-injectivity of the boundary components is preserved under
composition of cobordisms. This can be shown inductively by glueing
one pair of components at a time and applying the Seifert--van~Kampen
theorem as well as the normal form theorems for amalgamated free
products and HNN extensions~\cite[Chapter~11]{rotman}. These guarantee
at each stage that the remaining boundary components are
$\pi_1$-injective in the resulting manifold.

\smallskip

The symmetric monoidal pairing of~$\Cob_d$ clearly restricts to a
symmetric monoidal pairing on~$\Gcob d$. When $G = \Am$ is the class
of all amenable groups, we will refer to~$\Amcob d$ as the
\emph{amenable cobordism category}.

\subsubsection{Simplicial volume as a TQFT on the amenable cobordism category.} 
Let $\R = (\R, +)$ denote the additive (abelian) group of real
numbers, regarded as a symmetric monoidal groupoid with one object.
Moreover, let $G$ be a class of amenable groups that is closed under
isomorphisms. The additivity of the simplicial volume with respect to
amenable glueings (Theorem~\ref{thm:sv:additivity}) and disjoint union
shows that we obtain a symmetric monoidal functor with values in the abelian 
group~$\R$ (regarded as a symmetric monoidal category):
\begin{align*}
  \sv{-} \colon \Gcob d
  \to \R,
  \ (W; M, N) \mapsto \sv{W, \partial W}.
\end{align*}
In other words, the simplicial volume defines a TQFT on~$\Gcob d$.
Because this TQFT takes values in an abelian group (hence Picard groupoid), it is invertible.

\subsubsection{The fundamental group of~$B\Gcob d$.}
Let $B \Gcob d$ denote the classifying space of the cobordism
category~$\Gcob d$.  An object~$M$ of~$\Gcob d$ determines a point
($0$-simplex) $[M] \in B \Gcob d$ and we denote by~$\Omega_M B\Gcob d$
the loop space of the classifying space~$B\Gcob d$ based
at~$[M]$. Note that the monoid of path-components of~$B\Gcob d$ is a
group (similarly to~$B\Cob_d$).  Thus, $B \Gcob d$ is an infinite loop
space, therefore, all of its path components have the same homotopy
type. After passing to the classifying spaces, the functor~$\sv{-}$
induces a group homomorphism:
\begin{align*}
\phi_M \colon \pi_1(B\Gcob d, [M])  \cong \pi_0(\Omega_M B \Gcob d)
  & \to \pi_0(\Omega B \R) \cong \R.
\end{align*}
Here, we have used the homotopy equivalence~$\Omega B \Gamma \simeq
\Gamma$ for groups~$\Gamma$. The group homomorphisms~$\phi_M$ (for all
basepoints~$M$) uniquely determine the functor~$\sv{-}$; similar facts
hold more generally for functors whose target is a groupoid~(see, for
example,~\cite{DGPinvfield}).

\begin{rem}\label{rem:tqftim}
  Every endomorphism~$(W; M, M)$ in~$\Gcob d$ defines an element $[W]
  \in \pi_0(\Omega_{M} B\Gcob d)$ whose image under the group
  homomorphism~$\varphi_M$ is the relative simplicial volume~$\sv{W,
    \partial W}$. In particular, if $[W] = [W']$, then $\sv{W,
    \partial W} = \sv{W', \partial W'}$.
\end{rem}

\begin{thm}\label{thm:cob}
  Let $G \subset \Am$ be a class of groups that is closed under
  isomorphisms and let $M$ be an object of~$\Gcob 4$. Then the
  group~$\pi_1(B \Gcob 4, [M])$ is \emph{not} finitely generated.
\end{thm}
\begin{proof}
  The relative simplicial volume of $4$-manifolds induces a group
  homomorphism
  \[ S \colon \pi_1(B \Gcob 4, [\varnothing]) \longrightarrow \R.
  \]
  The image of this group homomorphism contains the subset~$\SV(4)$
  (Remark~\ref{rem:tqftim}), which contains an infinite family of
  elements that are linearly independent over~$\Q$
  (Remark~\ref{rem:sv4nonfin}).  Therefore, the abelian group~$\im S$
  is not finitely generated and so $\pi_1(B \Gcob 4, [\varnothing])$
  is not finitely generated.

  As explained above, $\pi_1(B\Gcob 4, [M])$ is independent of the
  choice of basepoint~$[M]$, so the result follows.
\end{proof}

\begin{rem}
  We also expect corresponding results in higher dimensions.
  However, currently, not enough is known about the structure
  of~$\SV(d)$ for~$d \geq 5$.
\end{rem}

\subsubsection{Non-extendability to~$\Cob_d$.} 
Since the simplicial volume does not satisfy additivity in
general~\cite[Remark~7.9]{Frigerio}, it does not define a functor
on~$\Cob_d$.  However, it is still interesting to ask whether there
might be a different extension of the simplicial volume to general oriented 
compact manifolds with boundary which is always
additive. This question is closely related to the problem of extending
the functor~$\sv{-} \colon \Gcob d \to \R$ to the whole cobordism
category~$\Cob_d$. Based on the classification of functors with values
in a groupoid (see, e.g., \cite{DGPinvfield}), this problem is
essentially equivalent to the question of extending the
homomorphism~$\phi_\emptyset$ to $\pi_1(B\Cob_d, \varnothing)$. 

In contrast to~$\pi_1(B\Gcob 4, [M])$ (Theorem~\ref{thm:cob}), the
fundamental group of the $d$-dimensional cobordism category~$\pi_1(B\Cob_d, \varnothing)$ is well known and
simpler to describe. We first note that it agrees with the fundamental
group of the standard \emph{topologised} cobordism
category~$\mathscr{C}_d$~\cite[Section~2.4]{DGPinvfield}. This is
again independent of the choice of basepoint and can be identified
with the Reinhart bordism
group~$\mathfrak{R}_d$~\cite{Reinhart}\cite[Appendix A]{Ebert13}. We
recall that $\mathfrak{R}_d$ can be described as the group of
equivalence classes of oriented closed $d$-manifolds where the
equivalence relation is defined by cobordisms whose tangent bundle is
equipped with a nowhere-vanishing vector field that extends the normal
fields on the boundary components. This refined bordism group is known
to be a split extension of the usual oriented bordism
group~$\Omega_d^{\SO}$ by a cyclic group whose generator is
represented by the $d$-sphere. More precisely, there is a split exact
sequence:
\begin{equation} \label{reinhart}
0 \to \mathbb{Z}/\mathrm{Eul}_{d+1} \xrightarrow{[1] \mapsto [S^d]} \mathfrak{R}_d \to \Omega_d^{\SO} \to 0 \tag{$*$}
\end{equation}
where $\mathrm{Eul}_{n} = \{0\}$ if $n$ is odd, $\mathrm{Eul}_n =
2\mathbb{Z}$ if $n \equiv 2 \pmod 4$, and $\mathrm{Eul}_n =
\mathbb{Z}$ if $n$ is a multiple of~$4$. We refer to the
literature~\cite{Reinhart}\cite[Appendix A]{Ebert13}\cite{BS14} for
the properties of the bordism group~$\mathfrak{R}_d$ and the
description of the homotopy groups of~$B \mathscr{C}_d$ in terms of
bordism classes.

Using this description of~$\mathfrak{R}_d \cong \pi_1(B \Cob_d,
[\varnothing])$, we conclude below that the simplicial volume of
oriented closed $d$-manifolds cannot be extended to a functor on the
cobordism category~$\Cob_d$, i.e., there is no additive extension of
the simplicial volume~$\sv{-}$ (analogous to
Theorem~\ref{thm:sv:additivity}) to all oriented compact $d$-manifolds.

Let $\mathcal{M}_d$ denote the monoid of endomorphisms
of~$\varnothing$ in~$\Cob_d$, that is, the monoid of diffeomorphism
classes of oriented closed $d$-manifolds under the operation of
disjoint union.

\begin{prop} \label{prop:tqft} 
  Let $d \geq 2$.
  \begin{itemize}
  \item[(1)] There is no functor~$\Cob_d \to \R$ that extends the
    restriction of the simplicial volume~$\sv{-}_{|\mathcal{M}_d}
    \colon \mathcal{M}_d \to \R$ to oriented closed $d$-manifolds.
  \item[(2)] Let $G \subset \Am$ be a class of groups that is closed
    under isomorphisms. The functor~$\sv{-} \colon \Gcob d
    \longrightarrow \R$ does not admit an extension to a functor
    on~$\Cob_d$.
  \end{itemize}
\end{prop}
\begin{proof}
  For~(1), note that such an extension of~$\Cob_d \to \R$ would imply
  a factorization of~$\sv{-}_{|\mathcal{M}_d} \colon \mathcal{M}_d \to
  \R$ through~$\mathfrak{R}_d \cong \pi_1(B\Cob_d, [\varnothing])$
  (see Remark~\ref{rem:tqftim}). In particular, this would imply that
  $\sv{-}$ is invariant under the Reinhart bordism relation. Moreover,
  since $\sv{S^d}=0$, it would further follow from the exact sequence
  \eqref{reinhart} that $\sv{-}$ is invariant under oriented
  bordism. This is obviously false in general, e.g., note that $M
  \sqcup (-M)$ is null-bordant as oriented closed $d$-manifold, but
  its simplicial volume is non-trivial in general. Claim~(2) follows
  directly from~(1).
\end{proof}

\begin{rem}
  The fact that $\sv{-}$ is not invariant under oriented bordism can
  also be shown as follows. Note that for~$d \in \{2, 3\}$, this fails
  because $\sv{-}$ is non-trivial but $\Omega^{\SO}_d \cong 0$. Then
  the result follows in all dimensions by taking suitable products. We
  note also that for~$d = 4$, this property can be shown to fail also
  because the oriented bordism group~$\Omega^{\SO}_4 \cong \Z$ is
  finitely generated, whereas $\SV(4)$ is not finitely generated by
  Remark~\ref{rem:sv4nonfin}.
\end{rem}

\begin{rem}\label{rem:eulerTQFT}
  In contrast to simplicial volume, the (relative) Euler
  characteristic defines a (symmetric monoidal) functor~$\chi \colon
  \Cob_d \to \Z$ (invertible TQFT), which sends $(W; M, N)$
  to~$\chi(W, M)$. Indeed, the Euler characteristic is invariant under
  the Reinhart bordism relation~\cite{Reinhart}.
\end{rem}

\section{Asphericalisations}\label{sec:asphericalisation}

The construction of aspherical closed manifolds with vanishing
simplicial volume is a key problem for
Question~\ref{quest:gromov}. There are several known constructions of
aspherical closed manifolds from non-aspherical or non-closed
manifolds. Important examples of such constructions are Davis'
reflection group trick~\cite{Davisannals, Davis:trick:proceedings,
  Davisbook} and Gromov's hyperbolization~\cite{Gromov1987,DavisJ,
  DavisJW}. The general difficulty with using these constructions to
obtain (counter)examples to Question~\ref{quest:gromov} has to do with
the difficulty of computing the simplicial volume of the resulting
aspherical closed manifolds.

In this section, we consider extensions of the class of aspherical
closed manifolds and look for interesting (counter)examples in these
contexts.  In particular, we will prove that the class of aspherical
spaces that are homology equivalent to closed manifolds, as well as
the class of closed manifolds that are homology equivalent to an
aspherical space, do not satisfy \pref{p:svchi} in general
(Theorem~\ref{thm:homologymfd}). We introduce the simplicial volume
of such spaces in Section~\ref{subsec:svhomologymfd}. The proof of 
Theorem~\ref{thm:homologymfd} will be given in
Section~\ref{subsec:kanthurston}; the proof is based on the Kan--Thurston theorem
\cite{Kan-Thurston}, which we recall in
Section~\ref{subsec:acyclicmaps}.
Finally, we end with some brief comments on known constructions of
aspherical closed manifolds and their possible connections with
Question~\ref{quest:gromov}
(Section~\ref{subsec:other_aspher}). Besides their independent
interest, we hope that the results of this section, especially,
combined with the aforementioned constructions of aspherical closed
manifolds, might provide useful tools for promoting non-aspherical or
non-closed examples to closed aspherical (counter)examples to
Question~\ref{quest:gromov}.

\subsection{Simplicial volume of spaces homology equivalent to manifolds}\label{subsec:svhomologymfd}

Our goal in this section is to extend the definition of simplicial
volume to spaces that are only homology equivalent to an oriented
closed manifold and discuss the main properties of this invariant.
This is motivated by the following basic observation:

\begin{rem}\label{key_rem}
  Suppose that $M$ is an oriented closed connected manifold
  with~$\sv{M} = 0$. Let $f \colon M \to N$ be a homology equivalence
  to an oriented closed connected manifold~$N$; in particular, this
  map has degree~$\pm 1$ and so~$\sv N = 0$ (this conclusion holds more generally 
  if the degree of $f$ is non-zero). Moreover, because 
  $f$ is a homology equivalence, it follows that $\chi(M) = \chi(N)$.  In this sense,
  \pref{p:svchi} is inherited under homology equivalences 
  between oriented closed connected manifolds. 
  
 Thus, in connection with Question~\ref{quest:gromov}, it would be interesting to understand the class of manifolds 
 which are homology equivalent to an oriented closed aspherical manifold with 
  vanishing simplicial volume. 
\end{rem}

\begin{defi} \label{def_gen_simp_vol}
  Let $X$ be a topological space, let $M$ be an oriented closed
  connected $n$-manifold and let $f \colon X \to M$ be an integral
  homology equivalence. We define the ($\R$-)\emph{fundamental class of~$(X,
    f)$} by
  \[ [X]_f := H_n(f;\R)^{-1} ([M]) \in H_n(X;\R)
  \]
  and the \emph{simplicial volume of~$X$} by
  \[ \sv X := \bigl\| [X]_f \bigr\|_1 \in \R_{\geq 0}.
  \]
\end{defi}

\begin{rem}
  The simplicial volume of such spaces is well-defined in the
  following sense. Let $(X, M^n, f)$ be as above.  In particular,
  $H_k(X;\Z)$ vanishes for~$k > n$ and $H_n(X;\Z) \cong H_n(M;\Z)
  \cong \Z$. Therefore, $H_n(f;\Z)^{-1}([M]_{\Z})$ is one of the two
  generators of~$H_n(X;\Z)$, which only differ by a sign.  In
  particular, the $\R$-fundamental class of~$(X, f)$ is independent
  of~$M$ and~$f$ up to sign. Therefore, the simplical volume of~$X$ is
  independent of the choice of~$M$ and~$f$. Clearly the definition of $\sv X$
  applies more generally whenever the map $f \colon X \to M$ induces an isomorphism on $H_n(-;\Z)$.
\end{rem}

\begin{rem}[Degree estimate]\label{rem:homologymfddeg}
  Let $(X, M, f)$ and $(Y, N, g)$ be as in
  Definition~\ref{def_gen_simp_vol}, where $M$ and $N$ are oriented
  closed connected manifolds of the same dimension~$n$. If $h \colon X
  \longrightarrow Y$ is a continuous map, then the \emph{unsigned
    homological degree}~$\mathopen|\deg h|$ is defined to be the
  unique natural number~$d \in \N$ with
  \[ H_n(h;\R)[X]_f = \pm d \cdot [Y]_g \in H_n(Y;\R). 
  \]
  As in the manifold case, we clearly have
  \[ \mathopen|\deg h| \cdot \sv Y \leq \sv X
  \]
  and it follows that the simplicial volume of~$X$ is homotopy
  invariant. Moreover, if $X$ admits a self-map~$h \colon X
  \longrightarrow X$ with~$\mathopen|\deg h| \geq 2$, then $\sv X =
  0$.  Furthermore, if $h \colon X \to Y$ is a (finite) covering map,
  then $\mathopen| \deg h| \cdot \sv{Y} = \sv{X}$, as can be seen from
  the same argument as in the manifold case.
\end{rem}

This extension of the simplicial volume to a homotopy invariant of
spaces that are only homology equivalent to an oriented closed
manifold should not be confused with the fact that the simplicial
volume is \emph{not} invariant under homology equivalences:

\begin{example}
  There exist oriented closed connected non-positively curved (and
  hence aspherical) homology $4$-spheres~$M$~\cite{RatcliffeT}; in
  particular, $M$ is homology equivalent to~$S^4$ and a result by
  Fujiwara and Manning~\cite[Corollary~2.5]{FMfilling} shows that $\sv
  M > 0 = \sv{S^4}.$
\end{example}

\subsection{Acyclic maps and plus contructions} \label{subsec:acyclicmaps}

We review briefly the definition and basic properties of acyclic maps
and refer to the literature~\cite{HHacyclic, Ra-acyclic} for more
details. A map~$f \colon X \longrightarrow Y$ is
\emph{acyclic} if the induced homomorphism
$$H_*(f; \mathcal{A}) \colon H_*(X;f^*\mathcal{A}) \longrightarrow
H_*(Y;\mathcal{A})$$ is an isomorphism for every local coefficient
system~$\mathcal{A}$ of abelian groups on~$Y$; in particular, $f$
induces isomorphisms on singular homology and cohomology with both
integral and real coefficients. Equivalently, a map~$f \colon X \to Y$
is acyclic if its homotopy fibers have trivial integral
homology. Every acyclic map~$f \colon X \to Y$ between path-connected
based spaces arises up to weak homotopy equivalence as the plus
construction~$\iota_P \colon X \to X^+_P$ with respect to a normal
perfect subgroup~$P \unlhd \pi_1(X)$. In this case, we have
$$\pi_1(X^+_P) \cong \pi_1(X)/P.$$

\begin{thm}[Kan--Thurston \cite{Kan-Thurston,BDH, Maunder}] \label{Kan-Thurston}
  For every path-connected based topological space~$X$, there is a
  group~$G_X$ together with an acyclic (based) map~$f_X \colon
  K(G_X,1) \to X$. Moreover, $G_X$ and $f_X$ can be chosen to be
  natural in~$X$.
\end{thm}
\begin{proof}
  The original functorial construction of~$(G_X, f_X \colon K(G_X,1)
  \to X)$ is due to Kan and Thurston \cite{Kan-Thurston}. Alternative
  constructions and refinements were obtained by Baumslag, Dyer,
  Heller~\cite{BDH}, and Maunder~\cite{Maunder}. (These constructions
  are also shown to preserve properties of homotopy finiteness, but
  they satisfy weaker functoriality properties in general.)
\end{proof}

\subsection{Using the Kan--Thurston theorem}\label{subsec:kanthurston}

The Kan--Thurston theorem (Theorem~\ref{Kan-Thurston}) has the
following consequence in connection with Question~\ref{quest:gromov}.

\begin{thm}\label{thm:homologymfd}
  Let $n \in \N_{\geq 2}$ be even. 
  \begin{enumerate}
  \item There exist aspherical CW-complexes~$X$ that admit an acyclic
    map $X \to M$ to an oriented closed connected $n$-manifold~$M$,
    and satisfy both $\sv{X} = 0$ and $\chi(X) \neq 0$. In particular,
    these aspherical spaces do \emph{not} satisfy \pref{p:svchi}.
  \item There exist oriented closed connected $n$-manifolds~$M$ that
    admit an acyclic map~$X \to M$ from an aspherical CW-complex~$X$,
    and satisfy both $\sv{M} = 0$ and $\chi(M) \neq 0$. In particular,
    these manifolds do \emph{not} satisfy \pref{p:svchi}.
  \end{enumerate}
\end{thm}

\begin{proof}
  Let $M$ be an oriented closed connected $n$-manifold that has a
  (based) self-map~$h \colon M \to M$ with~$\mathopen|\deg h |\geq 2$
  and satisfies~$\chi (M) \neq 0$. For example, as $n$ is even, we may
  choose~$M = S^n$.
  
  \emph{Ad~1}. By Theorem~\ref{Kan-Thurston} (and the functoriality of
  the construction), there exists an aspherical CW-complex~$X$ with an
  acyclic map $f \colon X \to M$ and a map~$H \colon X \to X$ that
  makes the following square commutative:
  $$
  \xymatrix{
    X \ar[d]_f \ar[r]^H & X \ar[d]^f \\
    M \ar[r]^h & M. 
  }
  $$
  It follows that $\mathopen| \deg H | = \mathopen|\deg h| \geq
  2$. Thus, $\sv X = 0$ (Remark~\ref{rem:homologymfddeg}). Moreover,
  as $f$ is a homology equivalence, it follows that $\chi(X) = \chi(M)
  \neq 0$.

  \emph{Ad~2}. It is sufficient to apply Theorem~\ref{Kan-Thurston} to~$M$.
\end{proof}

We give some further context on possible improvements of
Theorem~\ref{thm:homologymfd}:

\begin{rem}[Poincar\'e duality]
  Hausmann~\cite{Hausmann86} proved that the group~$G_X$ in the
  Kan--Thurston theorem (Theorem~\ref{Kan-Thurston}) can be chosen to
  be a duality group when $X$ is homotopy finite. (A related
  interesting refinement of the Kan--Thurston theorem has also been
  obtained more recently by R.~Kim~\cite{KimR}.) We do not know
  whether we can obtain homotopy finite examples in Theorem~\ref{thm:homologymfd} and whether
  Hausmann's construction can also be made sufficiently functorial for
  the purpose of the proof above. Thus, it remains open whether
  Theorem~\ref{thm:homologymfd} can be strengthened to produce
  examples where the fundamental group is a duality group.

  We note that it remains an open problem whether every finitely
  presented Poincar\'e duality group is the fundamental group of a
  closed aspherical manifold~\cite{Davis:PDG}. In this connection, we
  also recall the following question~\cite[Question~(ii)
    p.~254]{Kan-Thurston}: Is every oriented closed connected
  $n$-manifold, $n \geq 4$, homology equivalent to an oriented closed
  aspherical manifold? If this question has an affirmative
  answer, then there exist aspherical homology spheres in all high
  dimensions. This would contradict a version of the Hopf conjecture,
  which claims that the Euler characteristic of every
  oriented closed aspherical $2k$-manifold is either zero or its 
  sign is~$(-1)^k$~\cite{Davis:Hopf:Sign}.
\end{rem}

\subsection{Further comments}\label{subsec:other_aspher}

Davis' reflection group trick~\cite{Davisannals, Davis:trick:proceedings,
  Davisbook} takes an oriented compact aspherical
$n$-manifold~$(W, \partial W)$ and constructs an oriented
\emph{closed} aspherical $n$-manifold~$M$ by reflecting~$W$
along so-called \emph{pieces} of the boundary of~$W$. The construction
also yields a retraction
$$
r \colon M \to W \;
$$
i.e., $r \circ i = \id_W$, where $i \colon W \to M$ is the inclusion.
In particular, $\pi_1(r)$ and $H_*(r)$ are epimorphisms; $H^*_{b}(r)$
and $H^*(r)$ are monomorphisms. Starting from``exotic''~$W$, this
method can be used to construct ``exotic'' closed aspherical
manifolds. Note that we may choose~$W$ to be an oriented compact
manifold that is homotopy equivalent to~$X$ as in
Theorem~\ref{thm:homologymfd}(1), assuming that the space~$X$ can also
be chosen to be homotopy finite~\cite{BDH}.

There exists an explicit formula for computing the Euler
characteristic of the manifold~$M$ in terms of the Euler
characteristic of the input manifold~$W$ and the combinatorics of the
pieces of~$\partial W$~\cite[p. 218]{Davis:trick:proceedings}.
However, in the case of the simplicial volume, the situation is more
delicate. The Davis reflection group trick can be viewed as a refined
version of doubling manifolds with boundary, where the refinement is
given by the combinatorics of the pieces of~$\partial W$. In order to
generalise Example~\ref{exa:double:sv} to this setting, it would be
desirable to find input manifolds~$W$ with~$\sv{W,\partial W} = 0$ and
where additionally the simplicial volume of~$W$ can be realised by
small relative fundamental cycles, whose behaviour on~$\partial W$ is
adapted to the combinatorics on the pieces of~$\partial W$.
Particularly interesting input candidates would be the examples from
Remark~\ref{rem:svchibfails} or
Proposition~\ref{prop:vanishing:triple:product}.

\medskip

On the other hand, Gromov's hyperbolization~\cite{Gromov1987,DavisJ,
  DavisJW} is a construction that takes an oriented closed connected triangulated
manifold~$N$ and produces an oriented closed aspherical
manifold~$h(N)$ together with a degree~$1$ map
$$
c \colon h(N) \to N.
$$
In particular,
$$
\sv{h(N)} \geq \sv{N} \ .
$$

In addition, $h(N)$ is a smooth manifold if $N$ is smooth, and $h$
preserves the stable tangent bundle, i.e., the vector bundles
$T(h(N))$ and~$c^*TN$ are stably isomorphic. This implies that
hyperbolization preserves the characteristic classes and numbers of
closed smooth manifolds. Also, $h(N)$ and~$N$ are (oriented)
cobordant.  Since the mod~$2$ Euler characteristic of~$N$ is
determined by the bordism class of~$N$, it is natural to consider the
hyperbolization in connection with the following weak version of
Question~\ref{quest:gromov}:

\begin{quest}\label{quest:gromov_parity}
  Let $M$ be an oriented closed aspherical manifold. Does
  the following implication hold?
  \begin{equation}
    \sv M = 0  \Longrightarrow \chi(M) \text{ is even? } 
    \plabel{SV$\chi \mathrm{(mod 2)}$}{p:svchi2}
  \end{equation}
\end{quest}

Assuming that $M$ is smooth, the property
``$\chi(M)$ is even'' is equivalent to the vanishing of the top
Stiefel-Whitney class of~$M$. 

It would be interesting to find~$N$ as above with the
property~$\sv{h(N)} = 0$. We note here that the simplicial volume is
always positive in the case of \emph{strict} hyperbolization (in the
sense of Charney and Davis~\cite{charney-davis_strict}). A relative
version of this construction, which might still be relevant in
connection with Question~\ref{quest:gromov:relative}, has also been
studied by Belegradek~\cite{Belegradek:hyperbolization}.

\section{Stable integral simplicial volume}\label{sec:sisv}

Stable integral simplicial volume and integral foliated simplicial
volume are versions of the simplicial volume that admit Poincar\'e
duality estimates for Betti numbers and for the Euler characteristic.

In this section, we recall definitions, basic properties, and known
examples of stable integral simplicial volume, with a focus on the
relative case and the connection with \pref{p:svchib}. Moreover, we
quickly outline the relation with the integral foliated simplicial
volume.

\begin{defi}[Stable integral simplicial volume]
  Let $(M, \partial M)$ be an oriented compact connected manifold~$M$ with (possibly
  empty) boundary~$\partial M$.  The \emph{stable integral simplicial
    volume} of~$(M,\partial M)$ is defined as
  \[
  \sv{M, \partial M}_\mathbb{Z}^\infty \coloneqq
  \inf \Bigl\{\frac{\isv{N, \partial N}}{\mathopen|\deg f|} \Bigm|  (N,f) \in C(M) \Bigr \},
  \]
  where $C(M)$ denotes the class of all finite (connected) coverings
  of~$M$.
\end{defi}

\subsection{Estimates for the Betti numbers and the Euler characteristic}\label{subsec:betti}

The key observation is that Poincar\'e duality leads to Betti number
estimates for simplicial volumes with respect to sufficiently integral
coefficient rings:

\begin{prop}[\protect{\cite[Example~14.28]{lueck_l2}\cite[p.~307]{gromovmetric}\cite[Proposition~3.2]{kionkeloeh}}]\label{prop:betti}
  Let $R$ be a normed principal ideal domain with~$|x| \geq 1$ for
  all~$x \in R \setminus \{0\}$.  Let $(M, \partial M)$ be an oriented compact
  connected $n$-manifold with (possibly empty) boundary $\partial
  M$. Then, for all~$k\in \N$,
  \[ b_k(M;R) \leq \|M,\partial M\|_R
  \]
  In particular,
  $\bigl|\chi(M,\partial M)\bigr| = \bigl|\chi(M)\bigr|
     \leq (n+1) \cdot \|M,\partial M\|_R.
  $
\end{prop}

Since the Euler characteristic is multiplicative with respect to
finite coverings, this estimate also implies corresponding estimates
for the stable integral simplicial volume~\cite[Proposition~6.1]{FFM}:

\begin{cor}\label{cor:chi:sisv}
  Let $(M,\partial M)$ be an oriented compact connected $n$-manifold with
  (possibly empty) boundary $\partial M$. Then
  \[ \bigl| \chi(M,\partial M) \bigr| \leq (n + 1) \cdot \stisv {M,\partial M}.
  \]
\end{cor}

\subsection{Integral approximation problems}

This estimate for the Euler characteristic
(Corollary~\ref{cor:chi:sisv}) suggests the following question:

\begin{quest}\label{quest:integral:approx:absolute:version}
  Let $M$ be an oriented closed aspherical manifold with
  residually finite fundamental group. Does the following implication
  hold?
  \begin{equation}
    \sv M = 0 \Longrightarrow \stisv M = 0.
    \plabel{SV$_{\Z}^\infty$}{p:svapprox}
  \end{equation}
\end{quest}

The corresponding approximation question for non-zero values in
general has a negative answer. For example, oriented closed connected
hyperbolic manifolds~$M$ of dimension at least~$4$ satisfy~$\sv M <
\stisv M$~\cite{FFM}.  Moreover, Proposition~\ref{prop:connected:sum}
and Corollary~\ref{cor:chi:sisv} show that \pref{p:svapprox} does not
hold for the connected sum of oriented closed aspherical
manifolds with zero simplicial volume and zero Euler characteristic
(of even dimension at least~$4$).  However, these manifolds are never
aspherical (Remark~\ref{rem:connected:sum:not:asph}).

\medskip

In Question~\ref{quest:integral:approx:absolute:version}, one usually
adds the hypothesis of \emph{residual finiteness} to ensure the
existence of ``enough'' finite coverings. However, there are also no
known examples of oriented closed aspherical manifolds with
non-residually-finite fundamental group such that the vanishing
behaviour of the ordinary simplicial volume is different from that of
the stable integral simplicial volume. One possible strategy to
produce such examples is to use Davis' reflection group trick
(Section~\ref{subsec:other_aspher}) to construct oriented closed
aspherical manifolds whose fundamental group is not
residually finite. However, as explained in
Section~\ref{subsec:other_aspher}, it seems to be difficult to gain
enough control on the (stable integral) simplicial volume when
performing this construction.

\medskip

What about Question~\ref{quest:integral:approx:absolute:version} for
manifolds with boundary?  Similarly to \pref{p:svchib}, also in the
case of the stable integral simplicial volume, we need to impose
additional boundary conditions:

\begin{example}
  Let $M$ be the product of three punctured tori as in
  Example~\ref{example:triple:product:tori}. Then $M$ is aspherical
  and
  \[ \sv{M, \partial M} = 0 \quad \quad \mbox{ and } \quad \quad \chi(M, \partial M) = -1.
  \] 
  Hence, Corollary~\ref{cor:chi:sisv} implies that $\stisv{M,\partial
    M} \neq 0$, even though $\pi_1(M)$ is residually finite.
\end{example}

\begin{quest}\label{quest:int:approx:pi:1:injective}
  Let $M$ be an oriented compact aspherical manifold with
  residually finite fundamental group and non-empty $\pi_1$-injective
  aspherical boundary $\partial M$. Does the following implication hold?
  \begin{equation}
    \sv{M,\partial M} = 0 \Longrightarrow \stisv{M,\partial M} = 0.
    \plabel{SV$_\Z^\infty,\partial$}{p:svapproxb}
  \end{equation}
\end{quest}

\medskip

We observe the following diagram of implications as a consequence of
Corollary~\ref{cor:chi:sisv} and
Proposition~\ref{prop:gromov:implies:rel:gromov}:
\[ \xymatrix{%
  \text{\pref{p:svchi}} \ar@{=>}[r]
  & \text{\pref{p:svchib}}
  \\
  \text{\pref{p:svapprox}} \ar@{=>}[u]
  & \text{\pref{p:svapproxb}} \ar@{=>}[u]
  }
\]
We do not know whether the diagram can be completed with a lower
horizontal implication. The main issue is the lack of suitably general
additivity results concerning the integral simplicial volume.  For
example, given an oriented compact aspherical manifold~$M$
with residually finite fundamental group and non-empty
$\pi_1$-injective aspherical boundary, it is not clear what the
vanishing of~$\stisv{D(M)}$ has to say about the vanishing
of~$\stisv{M, \partial M}$.

Conversely, however, the vanishing of~$\stisv{M,\partial M}$ implies
the vanishing of~$\stisv{D(M)}$:

\begin{prop}\label{prop:sisv:relative:double}
  Let $(M, \partial M)$ be an oriented compact connected $n$-manifold with non-empty
  boundary $\partial M$. Then:
  \begin{enumerate}
  \item $\sv{M, \partial M}_\mathbb{Z}^\infty \geq \sv{\partial M}_\mathbb{Z}^\infty \slash (n+1)$;
  \item $\sv{D(M)}_\mathbb{Z}^\infty \leq 2 \cdot \sv{M, \partial
    M}_\mathbb{Z}^\infty$.
\end{enumerate}
\end{prop}
\begin{proof}
  \emph{Ad~1.} First recall that $\sv{N, \partial N}_\mathbb{Z} \geq
  \sv{\partial N}_\mathbb{Z} \slash (n+1)$ for every oriented compact
  connected $n$-manifold~$N$ with boundary~$\partial N$
  (Remark~\ref{rem:sv:with:int:coeff}).  Suppose that $\sv{M, \partial
    M}_\mathbb{Z}^\infty = T \in \, \R_{\geq 0}$, then for
  every~$\varepsilon > 0$ there exists an oriented compact connected
  finite covering~$N_\varepsilon$ of degree~$d_\varepsilon$ such that
  $$
  T \leq \frac{\sv{N_\varepsilon, \partial N_\varepsilon}_\mathbb{Z}}{d_\varepsilon} < T + \varepsilon \ .
  $$
  Hence, we also have
  $$
  \frac{1}{n+1} \cdot \frac{\sv{\partial N_\varepsilon}_\mathbb{Z}}{d_\varepsilon} < T + \varepsilon \ .
  $$
  Notice that the boundary~$\partial N_\varepsilon$ might consist of
  several different connected components~$S_1, \dots, S_k$ that
  cover~$\partial M$ with degrees $d_1, \dots, d_k$, respectively,
  such that $d_\varepsilon = \sum_{i = 1}^k d_i$.  By the pigeonhole
  principle, there exists a~$j \in \, \{1, \dots, k\}$ with
  $$\frac{1}{n+1} \cdot \frac{\sv{S_j}_\mathbb{Z}}{d_j} < T + \varepsilon \ ,$$
  because $\sum_{i=1}^k d_i/d_\varepsilon = 1$ and
  $$
  T + \varepsilon > \frac{1}{n+1} \cdot \frac{\isv{\partial N_\varepsilon}}{d_\varepsilon}
  = \sum_{i = 1}^k \frac{1}{n+1} \cdot \frac{\sv{S_i}_\mathbb{Z}}{d_\varepsilon}
  = \sum_{i = 1}^k \frac{d_i}{d_\varepsilon} \cdot \frac{1}{n+1} \cdot \frac{\sv{S_i}_\mathbb{Z}}{d_i} 
  \ .
  $$
  This shows that $\stisv{\partial M} \slash (n+1) \leq T +
  \varepsilon$. Letting $\varepsilon \to 0$ proves the claim.

  \emph{Ad~2.} If $(N, \partial N) \to (M, \partial M)$ is a finite
  covering of degree~$d$, then the induced map~$D(N) \to D(M)$ between
  the doubles is also a finite covering of degree~$d$. Moreover, by
  reflecting fundamental cycles (Example~\ref{exa:double:sv}), we have
  $\isv{D(N)} \leq 2 \cdot \isv{N, \partial N}$.
\end{proof}

\begin{rem}
  Let $(M,\partial M)$ be an oriented compact connected manifold with
  (possibly empty) boundary. Then
  Proposition~\ref{prop:sisv:relative:double} gives an alternative way
  to derive an Euler characteristic estimate for~$\stisv {M,\partial
    M}$ from the closed case. First, suppose that $n := \dim M$ is
  even.  In this case, $\chi(D(M)) = 2 \cdot \chi(M,\partial M)$.
  Hence, the closed case of Corollary~\ref{cor:chi:sisv} and the
  second part of Proposition~\ref{prop:sisv:relative:double} show that
  $$
  2 \cdot |\chi(M, \partial M)| = |\chi(D(M)| 
  \leq (n+1) \cdot \sv{D(M)}_\mathbb{Z}^\infty
  \leq 2 \cdot (n+1) \cdot \sv{M, \partial M}_\mathbb{Z}^\infty \ .
  $$
  Suppose now that $M$ has odd dimension~$n$. Then, we know that
  $\chi(\partial M) = 2 \cdot \chi(M)$.  Hence, by the closed case of
  Corollary~\ref{cor:chi:sisv} and the first part of
  Proposition~\ref{prop:sisv:relative:double}, we have
  \begin{align*}
    2 \cdot |\chi (M, \partial M)|
    & = 2 \cdot |\chi(M) - \chi(\partial M)| = |\chi(\partial M)| \\
    &\leq n \cdot \sv{\partial M}_\mathbb{Z}^\infty \leq n \cdot (n+1) \cdot \sv{M, \partial M}_\mathbb{Z}^\infty \ .
  \end{align*}
\end{rem}

\begin{example}
  For the following manifolds, the stable integral simplicial volume
  equals the classical simplicial volume; in particular, these
  examples satisfy \pref{p:svapprox} or \pref{p:svapproxb},
  respectively:

  \begin{enumerate}
  \item All oriented compact aspherical
    surfaces~\cite{vbc,kionkeloeh};
  \item All oriented compact aspherical $3$-manifolds with
    toroidal (or empty) boundary~\cite{FFL,FLMQ};
  \item All oriented closed connected generalised aspherical graph
    manifolds with residually finite fundamental group~\cite{FFL};
  \item All oriented closed aspherical manifolds with
    residually finite amenable fundamental group~\cite{FLPS};
  \item All oriented compact aspherical smooth manifolds with
    residually finite fundamental group admitting a non-trivial smooth
    $S^1$-action~\cite{FauserS1};
    \item All oriented compact aspherical smooth manifolds with
    residually finite fundamental group admitting an $F$-structure~\cite{LoehMoraschiniSauer};
  \item All oriented closed aspherical smooth manifolds with
    residually finite fundamental group admitting a regular circle
    foliation with finite holonomy groups~\cite{Camp-Corro};
    \item Every oriented closed aspherical manifold $M$ with
    residually finite fundamental group and $\amcat(M) \leq \dim(M)$~\cite{LoehMoraschiniSauer}.
    This applies e.g. to manifolds that are the total space of a fibre bundle 
    $M \to B$ with oriented closed connected fiber $N$ such that $\amcat(N) \leq \dim(M) \slash (\dim(B) +1)$
    and to manifolds of dimension $n \geq 4$ whose fundamental group $\Gamma$
    contains an amenable normal subgroup $A$ with $\cd_\mathbb{Z} (\Gamma \slash A) < n$.
  \end{enumerate}
\end{example}

\begin{rem}
  The arguments discussed in this section can be extended to principal
  ideal domains with norm bounded from below by~$1$.  Interesting
  examples of such rings include, for example, the finite fields with
  the trivial norm~\cite{Loehcoefficients}. In this setting, the proof
  of Proposition~\ref{prop:sisv:relative:double} also applies
  verbatim.
\end{rem}

\subsection{The dynamical version}

The Poincar\'e duality arguments also apply to the dynamical version
of the (integral) simplicial volume, i.e., to the integral foliated
simplicial volume:

\newcommand{\ifsv}[2][norm]{\!\csname #1l\endcsname\bracevert\!#2\!%
                            \csname #1r\endcsname\bracevert\!}

\begin{defi}[Integral foliated simplicial volume~\cite{gromovmetric,Sthesis,FFL}]
  Let $M$ be an oriented compact connected $n$-manifold with (possibly
  empty) boundary $\partial M$.
  \begin{itemize}
  \item If $\alpha = \pi_1(M) \actson (X,\mu)$ is a probability
    measure-preserving action on a standard Borel probability space,
    then we set
    \[ \ifsv {M,\partial M}^{\alpha}
    := \bigl\| [M,\partial M]^\alpha \|_1^\alpha,
    \]
    where $[M,\partial M]^\alpha \in H_n(M,\partial M
    ;L^\infty(X;\Z))$ denotes the image of the usual fundamental
    class~$[M,\partial M]_\Z$ under the inclusion of~$\Z$ into the
    \emph{twisted} coefficient module~$L^\infty(X;\Z)$. The
    norm~$|\cdot|_1^\alpha$ on the twisted chain complex is taken with
    respect to the $L^1$-norm on~$L^\infty(X;\Z)$.
  \item The \emph{(relative) integral foliated simplicial volume}
    of~$(M,\partial M)$ is defined as
    \[ \ifsv {M,\partial M} := \inf_{\alpha \in A(\pi_1(M))} \ifsv {M,\partial M}^\alpha,
    \]
    where $A(\pi_1(M))$ denotes the class of all probability
    measure-pre\-serv\-ing $\pi_1(M)$-actions on standard Borel
    probability spaces.
  \end{itemize}
\end{defi}

\begin{rem}
  For technical reasons, in the setting of manifolds with boundary, it
  is recommended to work with actions of the fundamental groupoid
  instead of the fundamental group~\cite{FFL,FLMQ}.
\end{rem}

\begin{prop}
  Let $(M,\partial M)$ be an oriented compact connected
  $n$-man\-i\-fold with (possibly empty) boundary $\partial M$. Then, for all~$k
  \in \N$,
  \[ \ltb k M \leq \ifsv {M,\partial M}.
  \]
  In particular, $\bigl|\chi(M,\partial M)\bigr| = \bigl|\chi(M)\bigr|
  \leq (n+1) \cdot \ifsv {M,\partial M}$.
\end{prop}
\begin{proof}
  This is a relative version of the $L^2$-Betti number estimate which
  was shown in the closed case by Schmidt~\cite{Sthesis} based on
  ideas of Gromov~\cite[p.~306]{gromovmetric}. When phrasing the proof
  in terms of $L^2$-Betti numbers of standard equivalence relations,
  literally the same proof as in the closed
  case~\cite[Theorem~6.4.5]{loeh_l2} can be applied to the twisted
  Poincar\'e--Lefschetz duality isomorphism.
\end{proof}

Many positive examples for Question~\ref{quest:gromov} have been
established using the integral foliated simplicial volume
(Section~\ref{subsec:relatedwork}) and most of the known computations
of the stable integral simplicial volume are based on ergodic
theoretic methods and the fact that
\[ \stisv {M,\partial M} = \ifsv{M,\partial M}^{\widehat{\pi_1(M)}}
\]
holds~\cite[Proposition~2.12]{FLMQ}\cite{LP,FLPS}, where
$\widehat{\pi_1(M)}$ denotes the dynamical system given by the
profinite completion of~$\pi_1(M)$.

A summary of computations of integral foliated simplicial volume and
of these ergodic theoretic methods can be found in the
literature~\cite[Chapter~6]{loeh_l2}.

\bibliographystyle{abbrv}
\bibliography{gromovbib}

\begin{thebibliography}{100}

\bibitem{Abrams}
L.~Abrams.
\newblock Two-dimensional topological quantum field theories and {F}robenius
  algebras.
\newblock {\em J.~Knot Theory Ramifications}, 5(5):569--587, 1996.

\bibitem{agol_MO}
I.~Agol.
\newblock Degrees of self-maps of aspherical manifolds.
\newblock Contribution to a MathOverflow discussion,
  \textsf{https://mathoverflow.net/questions/42168/degrees-of-self-maps-of-aspherical-manifolds},
  2010.

\bibitem{bsfibre}
I.~Babenko and S.~Sabourau.
\newblock Minimal volume entropy and fiber growth.
\newblock arXiv:2102.04551, 2021.

\bibitem{balacheffkaram}
F.~Balacheff and S.~Karam.
\newblock Macroscopic {S}choen conjecture for manifolds with nonzero simplicial
  volume.
\newblock {\em Trans.\ Amer.\ Math.\ Soc.}, 372(10):7071--7086, 2019.

\bibitem{BDH}
G.~Baumslag, E.~Dyer, and A.~Heller.
\newblock The topology of discrete groups.
\newblock {\em J.~Pure and Appl.~Algebra}, 16(1):1--47, 1980.

\bibitem{Belegradek:hyperbolization}
I.~Belegradek.
\newblock {Aspherical manifolds, relative hyperbolicity, simplicial volume and
  assembly maps}.
\newblock {\em Algebr.\ Geom.\ Topol.}, 6(3):1341--1354, 2006.

\bibitem{BHrel}
I.~Belegradek and G.~C. Hruska.
\newblock Hyperplane arrangements in negatively curved manifolds and relative
  hyperbolicity.
\newblock {\em Groups Geom. Dyn.}, 7(1):13--38, 2013.

\bibitem{BCGminimal}
G.~Besson, G.~Courtois, and S.~Gallot.
\newblock {Volume et entropie minimale des espaces localement sym{\'e}triques}.
\newblock {\em Invent.~Math.}, 103(2):417--445, 1991.

\bibitem{BS14}
M.~B{\"o}kstedt and A.~M. Svane.
\newblock A geometric interpretation of the homotopy groups of the cobordism
  category.
\newblock {\em Algebr.\ Geom.\ Topol.}, 14(3):1649--1676, 2014.

\bibitem{Bouarich}
A.~Bouarich.
\newblock Exactitude {\`a} gauche du foncteur ${H}^n_b(-; \mathbb{R})$ de
  cohomologie born{\'e}e r{\'e}elle.
\newblock {\em Ann.\ Fac.\ Sci.\ Toulouse. $6^e$~s{\'e}rie}, 10(2):255--270,
  2001.

\bibitem{braunphd}
S.~Braun.
\newblock {\em Simplicial Volume and Macroscopic Scalar Curvature}.
\newblock PhD thesis, Karlsruher Institut f{\"u}r Technologie (KIT), 2018.
\newblock available online with DOI:10.5445/IR/1000086838.

\bibitem{Bucher:product}
M.~Bucher.
\newblock Simplicial volume of products and fiber bundles.
\newblock In {\em Discrete groups and geometric structures}, volume 501 of {\em
  Contemp.\ Math.}, pages 79--86. Amer.\ Math.\ Soc., Providence, RI, 2009.

\bibitem{BBFIPP}
M.~Bucher, M.~Burger, R.~Frigerio, A.~Iozzi, C.~Pagliantini, and M.~B.
  Pozzetti.
\newblock Isometric embeddings in bounded cohomology.
\newblock {\em J.~Topol.\ Anal.}, 6(1):1--25, 2014.

\bibitem{BCL}
M.~Bucher, C.~Connell, and J.~F. Lafont.
\newblock {Vanishing simplicial volume for certain affine manifolds}.
\newblock {\em Proc.\ Amer.\ Math.\ Soc.}, 146:1287--1294, 2018.

\bibitem{BFP3}
M.~Bucher, R.~Frigerio, and C.~Pagliantini.
\newblock The simplicial volume of 3-manifolds with boundary.
\newblock {\em J. Topol.}, 8(2):457--475, 2015.

\bibitem{bucherneofytidis}
M.~Bucher and C.~Neofytidis.
\newblock The simplicial volume of mapping tori of $3$-manifolds.
\newblock {\em Math.\ Ann.}, 376(3-4):1429--1447, 2020.

\bibitem{bucher_finiteness}
M.~Bucher-Karlsson.
\newblock Finiteness properties of characteristic classes of flat bundles.
\newblock {\em Enseign.\ Math.~(2)}, 53(1-2):33--66, 2007.

\bibitem{bucherfirst}
M.~Bucher-Karlsson.
\newblock Simplicial volume of locally symmetric spaces covered by
  {$SL_3\mathbb{R}/SO(3)$}.
\newblock {\em Geom.\ Dedicata}, 125:203--224, 2007.

\bibitem{bucherprod}
M.~Bucher-Karlsson.
\newblock The simplicial volume of closed manifolds covered by $ \mathbb{H}^
  2\times\mathbb{H}^2$.
\newblock {\em J. Topol.}, 1:584--602, 2008.

\bibitem{Camp-Corro}
C.~Campagnolo and D.~Corro.
\newblock Integral foliated simplicial volume and circle foliations.
\newblock arXiv:1910.03071, to appear in J.~Topol.\ Anal.

\bibitem{CLM}
P.~Capovilla, C.~L\"oh, and M.~Moraschini.
\newblock Amenable category and complexity.
\newblock arXiv:2012.00612, to appear in Algebr.~Geom.~Topol.

\bibitem{CDMW}
G.~R. Chambers, D.~Dotterrer, F.~Manin, and S.~Weinberger.
\newblock Quantitative null-cobordism.
\newblock {\em J.~Amer.\ Math.\ Soc.}, 31(4):1165--1203, 2018.
\newblock With an appendix by Manin and Weinberger.

\bibitem{charney-davis_strict}
R.~M. Charney and M.~W. Davis.
\newblock Strict hyperbolization.
\newblock {\em Topology}, 34(2):329--350, 1995.

\bibitem{cheegergromov}
J.~Cheeger and M.~Gromov.
\newblock {Collapsing {R}iemannian manifolds while keeping their curvature
  bounded.~I}.
\newblock {\em J.~Differential~Geom.}, 23(3):309--346, 1986.

\bibitem{connellwang}
C.~Connell and S.~Wang.
\newblock Positivity of simplicial volume for nonpositively curved manifolds
  with a {R}icci-type curvature condition.
\newblock {\em Groups Geom.\ Dyn.}, 13(3):1007--1034, 2019.

\bibitem{connellwang2}
C.~Connell and S.~Wang.
\newblock Some remarks on the simplicial volume of nonpositively curved
  manifolds.
\newblock {\em Math.\ Ann.}, 377:969--987, 2020.

\bibitem{Davisannals}
M.~W. Davis.
\newblock Groups generated by reflections and aspherical manifolds not covered
  by {E}uclidean space.
\newblock {\em Ann. of Math.}, 117(2):293--324, 1983.

\bibitem{Davis:trick:proceedings}
M.~W. Davis.
\newblock Coxeter groups and aspherical manifolds.
\newblock In I.~H. Madsen and R.~A. Oliver, editors, {\em Algebraic Topology
  Aarhus 1982}, pages 197--221, Berlin, Heidelberg, 1984. Springer Berlin
  Heidelberg.

\bibitem{Davis:Hopf:Sign}
M.~W. Davis.
\newblock {The Hopf Conjecture and the Singer Conjecture}.
\newblock In I.~Chatterji, editor, {\em {Guido's Book of Conjectures}},
  volume~40 of {\em {Monographie de L'Enseignement Math.}}, pages 80--82.
  Princeton University Press, 2008.

\bibitem{Davisbook}
M.~W. Davis.
\newblock {\em The Geometry and Topology of Coxeter Groups. (LMS-32)}.
\newblock Princeton University Press, 2012.

\bibitem{Davis:PDG}
M.~W. Davis.
\newblock Poincar\'{e} duality groups.
\newblock In {\em {Surveys on Surgery Theory (AM-145)}}, volume~1, pages
  167--194. Princeton University Press, 2014.

\bibitem{DavisJ}
M.~W. Davis and T.~Januszkiewicz.
\newblock {Hyperbolization of polyhedra}.
\newblock {\em J.~Differential~Geom.}, 34(2):347--388, 1991.

\bibitem{DavisJW}
M.~W. Davis, T.~Januszkiewicz, and S.~Weinberger.
\newblock {Relative Hyperbolization and Aspherical Bordisms: An Addendum to
  {``Hyperbolization of Polyhedra''}}.
\newblock {\em J.~Differential~Geom.}, 58(3):535--541, 2001.

\bibitem{DGPinvfield}
A.~Debray, S.~Galatius, and M.~Palmer.
\newblock Lectures on invertible field theories.
\newblock arXiv:1912.08706, 2019.

\bibitem{dodziuk}
J.~Dodziuk.
\newblock {$L^{2}$} harmonic forms on rotationally symmetric {R}iemannian
  manifolds.
\newblock {\em Proc.\ Amer.\ Math.\ Soc.}, 77(3):395--400, 1979.

\bibitem{Ebert13}
J.~Ebert.
\newblock A vanishing theorem for characteristic classes of odd-dimensional
  manifold bundles.
\newblock {\em J.~Reine Angew.\ Math.}, 684:1--29, 2013.

\bibitem{Edmonds}
A.~L. Edmonds.
\newblock Aspherical 4-manifolds of odd {E}uler characteristic.
\newblock {\em Proc.\ Amer.\ Math.\ Soc.}, 148:421--434, 2020.

\bibitem{FauserT}
D.~Fauser.
\newblock {\em Integral foliated simplicial volume and {$S^1$}-actions}.
\newblock PhD thesis, Universi{\"a}t Regensburg, 2019.
\newblock available online with DOI:10.5283/epub.40431.

\bibitem{FauserS1}
D.~Fauser.
\newblock Integral foliated simplicial volume and {$S^1$}-actions.
\newblock {\em Forum Math.}, 33(3):773--788, 2021.

\bibitem{FFL}
D.~Fauser, S.~Friedl, and C.~L{\"o}h.
\newblock Integral approximation of simplicial volume of graph manifolds.
\newblock {\em Bull.\ Lond.\ Math.\ Soc.}, 51(4):715--731, 2019.

\bibitem{FLMQ}
D.~Fauser, C.~L{\"o}h, M.~Moraschini, and J.~P. Quintanilha.
\newblock Stable integral simplicial volume of $3$-manifolds.
\newblock {\em J.~Topol.}, 14(2):608--640, 2021.

\bibitem{FLMbdd}
F.~Fournier-Facio, C.~L{\"o}h, and M.~Moraschini.
\newblock Bounded cohomology of finitely presented groups: {V}anishing,
  non-vanishing and computability.
\newblock arXiv:2106.13567, 2021.

\bibitem{FFM}
S.~Francaviglia, R.~Frigerio, and B.~Martelli.
\newblock Stable complexity and simplicial volume of manifolds.
\newblock {\em J.~Topol.}, 5:977--1010, 2012.

\bibitem{Frigerio:ell1}
R.~Frigerio.
\newblock Amenable covers and $\ell^1$-invisibility.
\newblock arXiv:1907.10547, to appear in J.~Topol.\ Anal.

\bibitem{Frigerio}
R.~Frigerio.
\newblock {\em Bounded cohomology of discrete groups}, volume 227.
\newblock American Mathematical Soc., 2017.

\bibitem{FLPS}
R.~Frigerio, C.~L{\"o}h, C.~Pagliantini, and R.~Sauer.
\newblock Integral foliated simplicial volume of aspherical manifolds.
\newblock {\em Israel J.\ Math.}, 216:707--751, 2016.

\bibitem{FM:Grom}
R.~Frigerio and M.~Moraschini.
\newblock Gromov's theory of multicomplexes with applications to bounded
  cohomology and simplicial volume.
\newblock arXiv:1808.07307, to appear in Mem.\ Amer.\ Math.\ Soc.

\bibitem{FMfilling}
K.~Fujiwara and J.~Manning.
\newblock {Simplicial volume and fillings of hyperbolic manifolds}.
\newblock {\em Algebr.\ Geom.\ Topol.}, 11(4):2237--2264, 2011.

\bibitem{FM}
K.~Fujiwara and J.~K. Manning.
\newblock Simplicial volume and fillings of hyperbolic manifolds.
\newblock {\em Algebr.\ Geom.\ Topol.}, 11:2237--2264, 2011.

\bibitem{GGW}
J.~G\'{o}mez-Larra{\~n}aga, F.~Gonz\'{a}lez-Acu{\~n}a, and W.~Heil.
\newblock Amenable category of three-manifolds.
\newblock {\em Algebr.\ Geom.\ Top.}, 13:905--925, 2013.

\bibitem{vbc}
M.~Gromov.
\newblock Volume and bounded cohomology.
\newblock {\em Publ.\ Math.\ Inst.\ Hautes \'Etudes Sci.}, 56:5--99, 1982.

\bibitem{Gromov1987}
M.~Gromov.
\newblock {\em Hyperbolic Groups}, pages 75--263.
\newblock Springer, New York, 1987.

\bibitem{gromovasym}
M.~Gromov.
\newblock Asymptotic invariants of infinite groups.
\newblock In {\em Geometric group theory, {V}ol.~2 ({S}ussex, 1991)}, volume
  182 of {\em London Math.\ Soc.\ Lecture Note Ser.}, pages 1--295. Cambridge
  Univ. Press, Cambridge, 1993.

\bibitem{gromovmetric}
M.~Gromov.
\newblock {\em Metric structures for {R}iemannian and non-{R}iemannian spaces}.
\newblock Modern Birkh\"{a}user Classics. Birkh\"{a}user Boston, Inc., Boston,
  MA, english edition, 2007.
\newblock Based on the 1981 French original, With appendices by M.~Katz,
  P.~Pansu and S.~Semmes, Translated from the French by Sean Michael Bates.

\bibitem{Hausmann86}
J.-C. Hausmann.
\newblock Every finite complex has the homology of a duality group.
\newblock {\em Math.\ Ann.}, 275(2):327--336, 1986.

\bibitem{HHacyclic}
J.-C. Hausmann and D.~Husemoller.
\newblock Acyclic maps.
\newblock {\em Enseign.\ Math.~(2)}, 25(1--2):53--75, 1979.

\bibitem{heuerloeh_trans}
N.~Heuer and C.~L{\"o}h.
\newblock Transcendental simplicial volumes.
\newblock arXiv:1911.06386, to appear in Ann.\ Inst. Fourier, 2020.

\bibitem{heuerloeh}
N.~Heuer and C.~L{\"o}h.
\newblock The spectrum of simplicial volume.
\newblock {\em Invent.~Math.}, 223:103--148, 2021.

\bibitem{inoueyano}
H.~Inoue and K.~Yano.
\newblock The {G}romov invariant of negatively curved manifolds.
\newblock {\em Topology}, 21:83--89, 1982.

\bibitem{Ivanov}
N.~V. Ivanov.
\newblock Foundations of the theory of bounded cohomology.
\newblock volume 143, pages 69--109, 177--178. 1985.
\newblock Studies in topology,~V.

\bibitem{Ivanov_open_covers}
N.~V. Ivanov.
\newblock Leray theorems for $l^1$-norms of infinite chains.
\newblock arXiv:2012.08690, 2020.

\bibitem{Ivanov_bac_covers}
N.~V. Ivanov.
\newblock Leray theorems in bounded cohomology theory.
\newblock arXiv:2012.08038, 2020.

\bibitem{ivanov_turaev}
N.~V. Ivanov and V.~G. Turaev.
\newblock The canonical cocycle for the {E}uler class of a flat vector bundle.
\newblock {\em Dokl.\ Akad.\ Nauk SSSR}, 265(3):521--524, 1982.

\bibitem{Kan-Thurston}
D.~M. Kan and W.~P. Thurston.
\newblock Every connected space has the homology of a~{$K(\pi, 1)$}.
\newblock {\em Topology}, 15(3):253--258, 1976.

\bibitem{kastenholzreinhold}
T.~Kastenholz and J.~Reinhold.
\newblock Essentiality and simplicial volume of manifolds fibered over spheres.
\newblock arXiv:2107.05892, 2021.

\bibitem{KimR}
R.~Kim.
\newblock {Every finite complex has the homology of some CAT$(0)$ cubical
  duality group}.
\newblock {\em Geometriae Dedicata}, 176:1--9, 2015.

\bibitem{kionkeloeh}
S.~Kionke and C.~L{\"o}h.
\newblock A note on $p$-adic simplicial volumes.
\newblock {\em Glasg.\ Math.~J}, 63(3):563--583, 2021.

\bibitem{Kock}
J.~Kock.
\newblock {\em Frobenius algebras and 2{D} topological quantum field theories},
  volume~59 of {\em London Mathematical Society Student Texts}.
\newblock Cambridge University Press, Cambridge, 2004.

\bibitem{lafontschmidt}
J.~F. Lafont and B.~Schmidt.
\newblock Simplicial volume of closed locally symmetric spaces of non-compact
  type.
\newblock {\em Acta Math.}, 197:129--143, 2006.

\bibitem{loeh_pp}
C.~L{\"o}h.
\newblock The {P}roportionality {P}rinciple of {S}implicial {V}olume.
\newblock diploma thesis, WWU~M{\"u}nster (arXiv:0504106), 2004.

\bibitem{Loehthesis}
C.~L{\"o}h.
\newblock {\em $\ell^1$-{H}omology and {S}implicial {V}olume}.
\newblock PhD thesis, WWU M{\"u}nster, 2007.
\newblock available online at
  \url{http://nbn-resolving.de/urn:nbn:de:hbz:6-37549578216}.

\bibitem{loeh_rg}
C.~L{\"o}h.
\newblock Rank gradient versus stable integral simplicial volume.
\newblock {\em Period.\ Math.\ Hung.}, 76:88--94, 2018.

\bibitem{loeh_cost}
C.~L{\"o}h.
\newblock Cost vs.\ stable integral simplicial volume.
\newblock {\em Groups.\ Geom.\ Dyn.}, 14(3):899--916, 2020.

\bibitem{loeh_l2}
C.~L{\"o}h.
\newblock {\em Ergodic theoretic methods in group homology}.
\newblock SpringerBriefs in Mathematics. Springer, 2020.
\newblock A minicourse on {$L^2$}-{B}etti numbers in group theory.

\bibitem{Loehcoefficients}
C.~L{\"o}h.
\newblock Simplicial volume with $\mathbb{F}_p$-coefficients.
\newblock {\em Period.\ Math.\ Hung.}, 80:38--58, 2020.

\bibitem{lmfibre}
C.~L\"oh and M.~Moraschini.
\newblock Topological volumes of fibrations: A note on open covers.
\newblock arXiv:2104.06038, to appear in Proc.\ Roy.\ Soc.\ Edinburgh Sect.~A,
  2021.

\bibitem{LoehMoraschiniSauer}
C.~L{\"o}h, M.~Moraschini, and R.~Sauer.
\newblock Amenable covers and integral foliated simplicial volume.
\newblock arXiv:2112.12223, to appear in New York J. Math., 2021.

\bibitem{LP}
C.~L{\"o}h and C.~Pagliantini.
\newblock Integral foliated simplicial volume of hyperbolic 3-mani\-folds.
\newblock {\em Groups Geom.\ Dyn.}, 10:825--865, 2016.

\bibitem{LS}
C.~L\"oh and R.~Sauer.
\newblock Bounded cohomology of amenable covers via classifying spaces.
\newblock {\em Enseign.\ Math.}, 66:147--168, 2020.

\bibitem{lueck_l2}
W.~L{\"u}ck.
\newblock {\em {$L^2$}-{I}nvariants: {T}heory and {A}pplications to {G}eometry
  and {$K$}-{T}heory}, volume~44 of {\em Ergebnisse der Mathematik und ihrer
  Grenzgebiete. 3. Folge. A Series of Modern Surveys in Mathematics [Results in
  Mathematics and Related Areas. 3rd Series. A Series of Modern Surveys in
  Mathematics]}.
\newblock Springer-Verlag, Berlin, 2002.

\bibitem{lueck_aspherical}
W.~L{\"u}ck.
\newblock Survey on aspherical manifolds.
\newblock In {\em European {C}ongress of {M}athematics}, pages 53--82. Eur.\
  Math.\ Soc., Z{\"u}rich, 2010.

\bibitem{MM}
S.~Matsumoto and S.~Morita.
\newblock Bounded cohomology of certain groups of homeomorphisms.
\newblock {\em Proc.~Amer.~Math.~Soc.}, 94(3):539--544, 1985.

\bibitem{Maunder}
C.~R.~F. Maunder.
\newblock {A short proof of a theorem of {K}an and {T}hurston}.
\newblock {\em Bull.~London~Math.~Soc.}, 13(4):325--327, 1981.

\bibitem{milnor_euler}
J.~Milnor.
\newblock On the existence of a connection with curvature zero.
\newblock {\em Comment.\ Math.\ Helv.}, 32:215--223, 1958.

\bibitem{Milnor-Stasheff}
J.~W. Milnor and J.~D. Stasheff.
\newblock {\em Characteristic classes}.
\newblock Annals of Mathematics Studies, No. 76. Princeton University Press,
  Princeton, N. J.; University of Tokyo Press, Tokyo, 1974.

\bibitem{mineyev}
I.~Mineyev.
\newblock Bounded cohomology characterizes hyperbolic groups.
\newblock {\em Q. J. Math.}, 53:59--73, 2002.

\bibitem{monodthompson}
N.~Monod.
\newblock Lamplighters and the bounded cohomology of {T}hompson's group.
\newblock {\em Geom. Funct. Anal.}, 32:662--675, 2022.

\bibitem{monodnariman}
N.~Monod and S.~Nariman.
\newblock Bounded and unbounded cohomology of homeomorphism and diffeomorphism
  groups.
\newblock arXiv:2111.04365, 2021.

\bibitem{BAc}
M.~Moraschini and G.~Raptis.
\newblock Amenability and acyclicity in bounded cohomology theory.
\newblock arXiv:2105.02821, 2021.

\bibitem{munkholm}
H.~J. Munkholm.
\newblock Simplices of maximal volume in hyperbolic space, {G}romov's norm, and
  {G}romov's proof of {M}ostow's rigidity theorem (following {T}hurston).
\newblock In {\em Topology {S}ymposium, {S}iegen 1979 ({P}roc.\ {S}ympos.,
  {U}niv.\ {S}iegen, {S}iegen, 1979)}, volume 788 of {\em Lecture Notes in
  Math.}, pages 109--124. Springer, Berlin, 1980.

\bibitem{paternainpetean}
G.~Paternain and J.~Petean.
\newblock {Minimal entropy and collapsing with curvature bounded from below}.
\newblock {\em Invent.~Math.}, 151:415--450, 2003.

\bibitem{Ra-acyclic}
G.~Raptis.
\newblock Some characterizations of acyclic maps.
\newblock {\em J.~Homotopy Relat.\ Struct.}, 14(3):773--785, 2019.

\bibitem{georgevanishing}
G.~Raptis.
\newblock Bounded cohomology and homotopy colimits.
\newblock arXiv:2103.15614, 2021.

\bibitem{RatcliffeT}
J.~G. Ratcliffe and S.~T. Tschantz.
\newblock Some examples of aspherical 4-manifolds that are homology 4-spheres.
\newblock {\em Topology}, 44(2):341--350, 2005.

\bibitem{Reinhart}
B.~L. Reinhart.
\newblock Cobordism and the {E}uler number.
\newblock {\em Topology}, 2:173--177, 1963.

\bibitem{rotman}
J.~J. Rotman.
\newblock {\em An introduction to the theory of groups}, volume 148 of {\em
  Graduate Texts in Mathematics}.
\newblock Springer-Verlag, New York, fourth edition, 1995.

\bibitem{sauerminvol}
R.~Sauer.
\newblock Amenable covers, volume and {$L^2$}-{B}etti numbers of aspherical
  manifolds.
\newblock {\em J.~Reine Angew.\ Math.}, 636:47--92, 2009.

\bibitem{Sthesis}
M.~Schmidt.
\newblock {\em $L^2$-Betti Numbers of $\mathcal{R}$-spaces and the Integral
  Foliated Simplicial Volume}.
\newblock PhD thesis, Westf{\"a}lische Wilhelms-Universit\"at M{\"u}nster,
  2005.
\newblock available online at
  \url{http://nbn-resolving.de/urn:nbn:de:hbz:6-05699458563}.

\bibitem{soma}
T.~Soma.
\newblock The {G}romov invariant of links.
\newblock {\em Invent.\ Math.}, 64(3):445--454, 1981.

\bibitem{Sullivan}
D.~Sullivan.
\newblock Infinitesimal computations in topology.
\newblock {\em Inst.\ Hautes {\'E}tudes Sci.\ Publ.\ Math.}, (47):269--331
  (1978), 1977.

\bibitem{thurston}
W.~P. Thurston.
\newblock The geometry and topology of $3$-manifolds.
\newblock mimeo\-graphed notes, 1979.

\bibitem{weinberger_computersrigidity}
S.~Weinberger.
\newblock {\em Computers, rigidity, and moduli}.
\newblock M.~B.\ Porter Lectures. Princeton University Press, 2005.
\newblock The large-scale fractal geometry of Riemannian moduli space.

\bibitem{wood_euler}
J.~W. Wood.
\newblock Bundles with totally disconnected structure group.
\newblock {\em Comment.\ Math.\ Helv.}, 46:257--273, 1971.

\bibitem{yano}
K.~Yano.
\newblock Gromov invariant and {$S^{1}$}-actions.
\newblock {\em J.~Fac.\ Sci.\ Univ.\ Tokyo Sect.\ IA Math.}, 29(3):493--501,
  1982.

\end{thebibliography}

\end{document}